\documentclass[12pt,a4paper]{article}
\usepackage{amsmath,amssymb,latexsym}
\usepackage{amsthm}
\usepackage{bm}
\usepackage{overpic}
\usepackage{color}
\usepackage{xcolor}

\usepackage{graphicx}
\graphicspath{{Figures/}}

\numberwithin{equation}{section} 

\theoremstyle{plain}
\newtheorem{theorem}{Theorem}[section]

\newtheorem{proposition}[theorem]{Proposition}
\newtheorem{lemma}[theorem]{Lemma}
\theoremstyle{definition}
\newtheorem{definition}[theorem]{Definition}

\theoremstyle{remark}
\newtheorem{remark}[theorem]{Remark}

\newcommand{\supp}{\mathop{\rm supp}}
\newcommand{\Bal}{\mathop{\rm Bal}}

\newcommand{\field}[1]{\mathbb{#1}}

\newcommand{\R}{\field{R}}

\newcommand{\N}{\field{N}}
\renewcommand{\H}{\field{H}}
\newcommand{\C}{\field{C}}

\newcommand{\CC}{{\mathcal C}}

\newcommand{\MM}{{\mathcal M}}

\renewcommand{\Re}{\mathop{\rm Re}}
\renewcommand{\Im}{\mathop{\rm Im}}

\def\XXint#1#2#3{{\setbox0=\hbox{$#1{#2#3}{\int}$}
\vcenter{\hbox{$#2#3$}}\kern-.5\wd0}}

\renewcommand{\bar}{\overline}
\renewcommand{\tilde}{\widetilde}
\renewcommand{\hat}{\widehat}
\newcommand{\OO}{\mathcal O}

\title{Equilibrium problems in weakly admissible external fields created by pointwise charges}

\author{R. Orive, J.F. S\'{a}nchez Lara and F. Wielonsky\footnote{Corresponding author}}

\date{\today}

\begin{document}

\vspace{1cm} \maketitle
\noindent
R. Orive, rorive@ull.es\\
 Departamento de An\'alisis Matem\'atico, Universidad de La Laguna, Spain.
\\[\baselineskip]
J. F. S\'anchez-Lara, jslara@ugr.es\\
 Departamento de Matem\'atica Aplicada, Universidad de Granada, Spain.
\\[\baselineskip]
F. Wielonsky, franck.wielonsky@univ-amu.fr\\
Laboratoire I2M, UMR CNRS 7373, Universit\'e Aix-Marseille\\
39 Rue Joliot Curie,
F-13453 Marseille Cedex 20, FRANCE
\newpage
\begin{abstract}
The main subject of this paper is equilibrium problems on an unbounded conductor $\Sigma$ of the real axis in the presence of a weakly admissible external field. An admissible external field $Q$ on $\Sigma$ satisfies, along with other mild conditions, the following growth property at infinity:
$$\lim_{|x| \rightarrow \infty}(Q(x) - \log |x|) = +\infty.$$
This condition guarantees the existence and uniqueness of the equilibrium measure in the presence of $Q$, and the compactness of its support. In the last 10-15 years, several papers have dealt with
weakly admissible external fields, in the sense that $Q$ satisfies a weaker condition at infinity, namely,
$$\exists M\in(-\infty,\infty],\quad\liminf_{|x| \rightarrow \infty}(Q(x) - \log |x|) = M.$$
Under this last assumption, there still exists a unique equilibrium measure in the external field $Q$, but the support need not be a compact subset of $\Sigma$ anymore. In most examples considered in the literature the support is indeed unbounded.
Our main goal in this paper is to illustrate this topic by means of a simple class of external fields on the real axis created by a pair of attractive and repellent charges in the complex plane, and to study the dynamics of the associated equilibrium measures as the strength of the charges evolves. As one of our findings, we exhibit configurations where the support of the equilibrium measure in a weakly admissible external field is a compact subset of the real axis. To achieve our goal, we extend some results from potential theory, known for admissible external fields, to the weakly admissible case. These new results may be of independent interest.  Finally, the so--called signed equilibrium measure is an important tool in our analysis. Its relationship with the (positive) equilibrium measure is also explored.
\end{abstract}
\section{Introduction}

Throughout this article, we consider equilibrium problems on the real axis in a general rational external field created by pointwise charges, of the form
\begin{equation}\label{fixedfield}
Q(x) = \sum_{j=1}^q\gamma_j\log|x-z_j|=-\sum_{j=1}^q\gamma_jV^{\delta_{z_j}} (x),
\end{equation}
where
$$\gamma_j\in \R,\quad z_j\in \C,\quad\sum_{j=1}^q\gamma_j = T>0,$$
and, as usual, the potential of a measure $\sigma$ is defined by
$$V^{\sigma}(x) = -\int \log |x-y|d\sigma (y),$$
and $\delta_z$ stands for the Dirac mass at a point $z\in \C$. Note that, if $\gamma_{k}>0$, the charge located at $z_{k}$ acts as an attractor while, if $\gamma_{k}<0$, the charge acts as a repellent.
For $\gamma_k>0,$ we assume that $z_k$ lies on $\C \setminus \R$ to ensure that the external field is a lower semi-continuous function. These external fields, or weights, are called rational since their derivatives are rational functions.
The particular case where $Q$ has an additional polynomial part of even degree
was considered in \cite{OrSL2015}. Then, in \cite{OrSL2016}, the case of rational external fields of type \eqref{fixedfield} (without polynomial part) was treated when all $\gamma_k > 0$ (and, thus, all $z_k \in \C \setminus \R$). One of the contributions of this paper is to deal with both positive and negative charges (that is, a mixture of attractors and repellents), where in principle, for $\gamma_k < 0$, $z_k \in \R$ is allowed.

Though, in general, the study of weighted equilibrium problems on a conductor $\Sigma$ deals with probabilistic equilibrium measures $\mu_Q$ with $\mu_Q (\Sigma) = 1,$ the authors of \cite{MOR2015}, \cite{OrSL2015} and \cite{OrSL2016} have preferred to consider equilibrium measures with varying mass $\mu_t = \mu_{t,Q}$, in the sense that $\mu_t (\Sigma) = t>0$. This is similar to the dynamical approach proposed in the seminal work by V. Buyarov and E. A. Rakhmanov \cite{BR:99}. Actually, both approaches are equivalent since
\begin{equation*}
\mu_{t,Q} = t\mu_{1,Q /t},
\end{equation*}
where $\mu_{1,Q /t}$ denotes the unit equilibrium measure in the external field $Q(x)/t.$

Following \cite{Saff:97}, an external field $Q$ for a (possibly unbounded) closed subset $\Sigma$ of the complex plane and for a given mass $t>0$ is said to be admissible if
\begin{itemize}

\item [(i)] $Q: \Sigma \rightarrow (-\infty,+\infty]$ is a lower semi-continuous function.

\item [(ii)] The set $\displaystyle \{x\in \Sigma:Q(x) < +\infty\}$ has positive (logarithmic) capacity.

\item [(iii)] In case $\Sigma$ is unbounded, $Q$ satisfies the growth condition at infinity:
\begin{equation}\label{condadmiss}
\lim_{|x| \rightarrow \infty}(Q(x) - t\log |x|) = +\infty.
\end{equation}

\end{itemize}

These conditions guarantee the existence of a unique measure $\mu_t = \mu_{t,Q}$, with total mass $t$, with compact support in $\Sigma$, minimizing the weighted energy
$$I_Q (\sigma) = -\iint \log |x-y|d\sigma (x) d\sigma (y)+2\int Q(x)d\sigma (x)=I(\sigma)+2\int Q(x)d\sigma (x),$$
where $I(\sigma)$ denotes the unweighted energy of $\sigma$. The equilibrium measure is uniquely characterized by the fact that its total (also called ``chemical'') potential satisfies the Frostman (or variational) inequalities:
\begin{align}\label{Frost1}
V^{\mu_t}(x)+Q(x) & \geq c_t,\quad q.e.~ x\in \Sigma,\\
\label{Frost2}
V^{\mu_t}(x)+Q(x) & \leq c_t,\quad x\in\supp \mu_t,
\end{align}
for some constant $c_t\in\mathbb{R}$ (the equilibrium constant), and where $q.e.$ means quasi-everywhere, that is, outside of a set of capacity 0. It follows from
\eqref{Frost1} and \eqref{Frost2} that
\begin{equation}\label{eq-ct}
c_{t}=\frac1t\left(I(\mu_{t})+\int Q(x)d\mu_{t}(x)\right).
\end{equation}

In the last years, a growing interest in the so-called {\it weakly admissible external fields} has taken place (see \cite{BLW}, \cite{Hardy-Kuijlaars} and \cite{Simeonov}, among others). For these external fields, condition \eqref{condadmiss} is replaced by the weaker growth assumption at infinity,
\begin{equation}\label{condwadmiss}
\liminf_{|x| \rightarrow \infty}(Q(x) - t\log |x|) = M,
\end{equation}
where $M\in(-\infty,\infty]$.
Note that $Q$ is lower bounded on $\Sigma$, and that, for any measure $\sigma$ of mass $t$, condition \eqref{condwadmiss} ensures that its weighted energy is well-defined:
$$I_{Q}(\sigma)=\iint\left(\log\frac{1}{|x-y|}+\frac{Q(x)}{t}+\frac{Q(y)}{t}\right)d\sigma(x)d\sigma(y)>-\infty.
$$
Then, the existence and uniqueness of the weighted equilibrium measure $\mu_{t}$ remain guaranteed, it has finite weighted energy $I_{Q}(\mu_{t})$ and finite energy $I(\mu_{t})$, see \cite[Theorem 3.4]{BLW},  but its support need not be a compact subset of $\Sigma$ (in fact, it is often an unbounded set in the examples considered in the literature up to now, as far as we know). {Observe that the balayage of a measure $\nu$ onto an unbounded closed set $\Sigma$ is a typical example of this situation; indeed, it agrees with the equilibrium measure of $\Sigma$ in the external field $Q(x) = -V^{\nu}(x)$, in such a way that the total negative (attractive) mass equals the total positive mass to be spread through $\Sigma$ (the notion of balayage of a measure is recalled in Section \ref{Pot}).

Returning to external fields of the form (\ref{fixedfield}), we can thus study the
equilibrium measure $\mu_t = \mu_{t,Q}$, with total mass $t\in (0,T]$.
The case where $t\in (0,T)$ corresponds to the admissible setting (\ref{condadmiss}),
while the limit case $t=T$ is related to the weakly admissible one. In this last situation the total attractive mass of the fixed charges equals the total mass of the free charge distributed along the real axis. This last case is our main concern in this paper.

It is interesting to describe the different possible scenarios for the equilibrium measure in these rational external fields. Also noteworthy is the fact that the weakly admissible external fields considered in this paper may be seen as limits of admissible external fields and, thus, the tools used in the previous papers \cite{MOR2015}, \cite{OrSL2015} and \cite{OrSL2016} for admissible rational external fields will be useful. Finally, the {\it signed equilibrium measure}, which is easier to compute than the (positive) equilibrium measure and is closely related to it (see e.g. \cite{KD}), will be another important tool in our analysis.

The outline of the paper is as follows. In Section 2, making use of the notion of signed equilibrium measure, we give a simple condition for compactness of the support of the equilibrium measure in a weakly admissible external field of the general form \eqref{fixedfield}. In Section 3 we specialize to the case where a weak external field is created by a pair of ``attractor/repellent'' charges. Our goal is to describe the evolution of the support of the equilibrium measure as the strength of the repellent charge increases. Our main result is stated in
Theorem \ref{thm:main}. In the last section, we have gathered several extensions of classical results of potential theory to an unbounded setting as well as properties of equilibrium problems in weakly admissible external fields. These results are used in the previous sections. They may also be of independent interest.

\section{Weak external field created by a finite number of charges: A condition for compactness}
In this short section, we consider weakly admissible external fields on $\R$ of the general form \eqref{fixedfield}, with $\sum_{j=1}^q\gamma_j = T>0$, and the corresponding equilibrium measure $\mu_T$ of mass $T$.
We give a criterion for compactness of the support of $\mu_{T}$.
Our analysis is based on the {\it signed equilibrium measure}, see Section \ref{sgn-meas} for details.

Recall that a signed measure $\eta_t = \eta_{t,Q}$ is said to be the {\it signed equilibrium measure} of $\Sigma$ of mass $t>0$ if
$\eta_t (\Sigma) =t$ and
\begin{equation}\label{signedequil}
V^{\eta_t} (x) + Q(x) = C,\quad q.e.~x\in \Sigma.
\end{equation}
The signed equilibrium measure $\eta_t$ is unique, provided it exists, see e.g. \cite{BDS2009} for the case of Riesz kernels. Its usefulness relies on the fact that it is easier to compute than the positive equilibrium measure $\mu_t$ and provides important information about the latter. In particular, it is well known that, in the presence of an admissible external field $Q$, the support of $\mu_t$ is contained in the support of $\eta_t^+$, the positive part of $\eta_t$ in its Jordan decomposition, see \cite[Lemma 3]{KD}. In Proposition 4.11 below we prove that this result remains valid in the weakly admissible case as well.

Now, suppose that $\Im z_j \neq 0,j=1,\ldots,q,$ in \eqref{fixedfield}. It is immediate 
that the opposite of the balayage of the measure $-\sum_{j=1}^{q}\gamma_{j}\delta_{z_{j}}$ on $\R$ is a signed measure
\begin{equation*}
\eta_T = \eta_{T,Q}:= \sum_{j=1}^q\gamma_j\Bal(\delta_{z_j},\R)
\end{equation*}
which has total mass $T$ and satisfies the identity \eqref{signedequil} on $\R$. Therefore, $\eta_T$ is the signed equilibrium measure in the weakly admissible external field \eqref{fixedfield}; moreover, \eqref{zrealbalayage}
below yields the expression of its density:
\begin{equation}\label{signedfixed}
\eta'_T (x) = \frac{1}{\pi}\sum_{j=1}^q\gamma_{j}\frac{|\Im z_j|}{|x-z_j|^2},\;x\in \R.
\end{equation}
Now, as a consequence of the above expression 
and the aforementioned fact that $\supp \mu_T \subset \supp \eta^+_T$ (Proposition 4.11),
we have
\begin{theorem}\label{thm:compsupp}
Suppose that $\Im z_j \neq 0, j=1,\ldots,q.$ If
\begin{equation}\label{generalcondcomp}
\displaystyle \sum_{j=1}^q\gamma_j|\Im z_j| < 0,
\end{equation}
then $\supp \mu_T$ is a compact subset of the real axis.
\end{theorem}
\begin{proof}
Indeed, if (\ref{generalcondcomp}) is satisfied, \eqref{signedfixed} implies that $\eta'_T (x) < 0$ for $|x|$ large enough and, thus, $\supp \eta^+_T$ is a compact set. Since
$\supp \mu_T \subset \supp \eta^+_T$, the assertion in the theorem follows.
\end{proof}
As pointed out in the introduction, the previous theorem
provides a simple class of weakly admissible external fields which confine the corresponding equilibrium measures in compact supports, see Figure \ref{fig:compsuppgral} for an example.
\begin{figure}
    \begin{center}
        \includegraphics[scale=0.4]{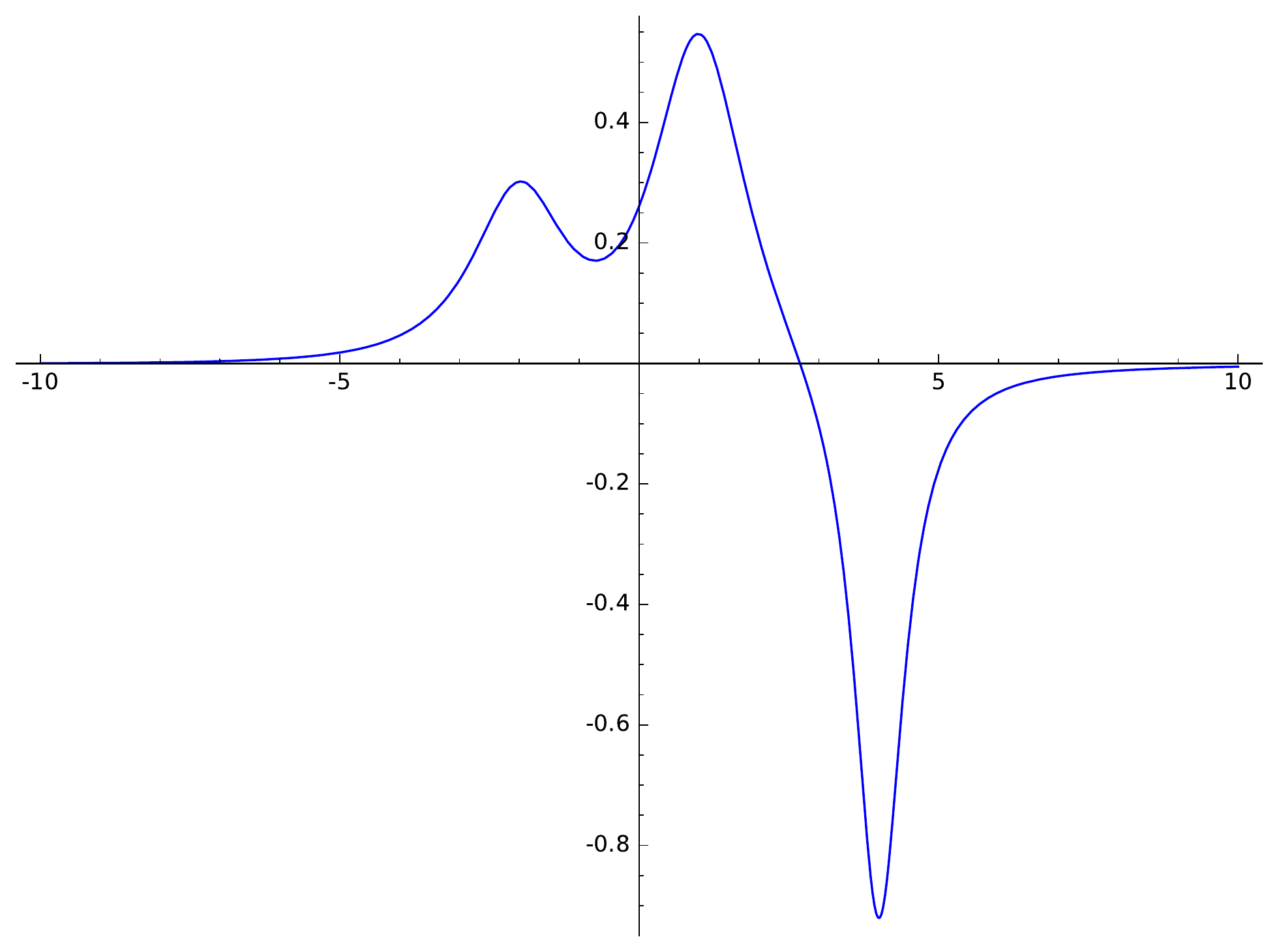}
    \end{center}
    \caption{Density of the signed equilibrium measure $\eta_{T}$, of total mass $0.5$, in
    a weak external field created by two attractors of charges $-1$ and $-2$, located at $z_1 = -2 + i$ and $z_3 = 1 + i$, and two repellents of charges $1$ and $1.5$, placed at $z_2 = 3i$ and $z_{4} = 4 + 0.5 i$.
Since $\eta_{T}(x)$ behaves like $-3/(4\pi x^{2})$ for $|x|$ large, the equilibrium measure $\mu_{T}$ has a compact support.}
    \label{fig:compsuppgral}
\end{figure}

Note that if all the charges are attractors (i.e. $\gamma_j > 0,j=1,\ldots,q$), then $\eta_T$ is a positive measure on the whole real axis. From the characterization of the equilibrium measure given in Proposition \ref{char-Frost} and the identity \eqref{signedequil}, we may thus deduce that $\mu_T = \eta_T$ is supported on the whole axis
$\R$. This was the case considered in the previous \cite{OrSL2016}.

\section{External field created by a pair of attractor and repellent charges}

In this section, we concentrate 
on external fields on $\R$ created by the combination of a negative (i.e. attractive) charge and a positive (i.e. repellent) charge of the form
\begin{equation}\label{def-pair}
-\delta_{z_{1}}+\gamma\delta_{z_{2}},
\end{equation}
with $z_{1}\in\C\setminus\R$, $z_{2}\in\C$, $\gamma \in (0,1)$. The corresponding external field is given by
\begin{equation}\label{twochargest}
\begin{split}
Q(x) & = \log |x-z_1| - \gamma \log |x-z_2| \\
& =\frac12\log(x-z_{1})+\frac12\log(x-\bar z_{1})-\frac{\gamma}{2} \log (x-z_2)- \frac{\gamma}{2} \log (x-\bar z_2),
\end{split}
\end{equation}
where in each of the logarithms in \eqref{twochargest} it is chosen the branch with a cut connecting $z_j$ or $\bar z_{j}$ with the point at infinity, and not intersecting the real axis.
The second expression in (\ref{twochargest}) gives the analytic extension of $Q$ outside of $\R$. We still denote it by $Q$.

For ease of exposition, we will assume that $\Re z_1 \neq \Re z_2$, hence, without loss of generality, we can suppose that
\begin{equation}\label{z1z2}
z_1 = -1 + \beta_1 i,\quad z_2 = 1 + \beta_2 i,\qquad\beta_{1},\beta_{2}>0.
\end{equation}
The two limit cases
$\Re z_1 = \Re z_2$ (e.g. $=0$) and $\beta_{2}=0$ will be considered separately  in Section \ref{limit-case}.}

We aim at studying the equilibrium measure $\mu_T$ of mass $T$, 
in the weakly admissible case $T=1-\gamma$, and our main concern is to describe how the support of the equilibrium measure evolves as the parameter $\gamma$ moves from 0 to 1. In particular, we will see that the support may be bounded or unbounded, and consists of a single interval (one-cut phase) or two  intervals (two-cut phase), see Theorem \ref{thm:main} below. It is also noteworthy that even with a single attractor as in the present case, the external field \eqref{twochargest} may have two minima,
see Figure \ref{fig:example1} for an example.
\begin{figure}
    \begin{center}
        \includegraphics[scale=0.4]{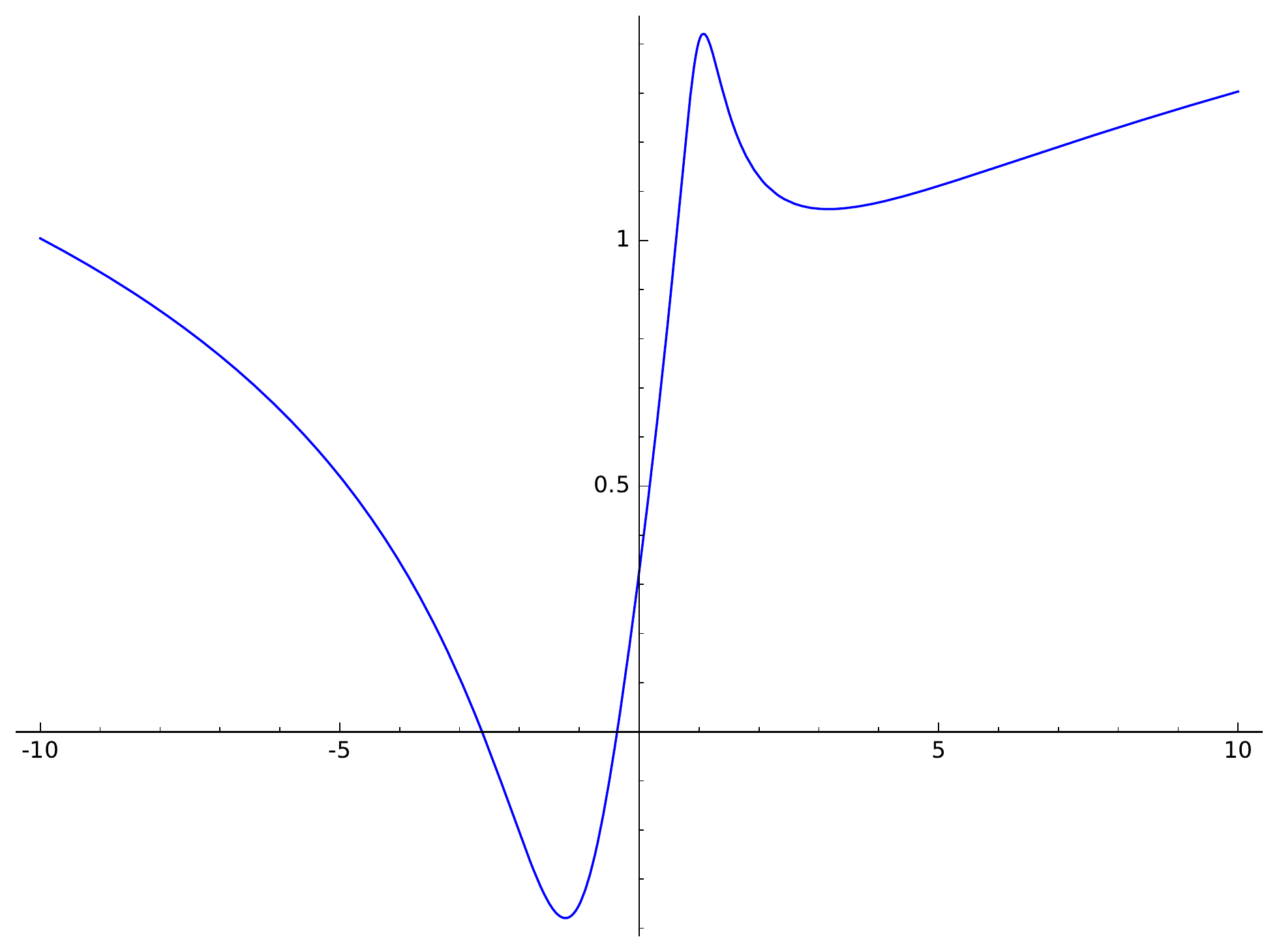}
    \end{center}
    \caption{An external field $Q$ due to an attractor/repellent pair with two minima. Here, $\beta_1 = 1,\beta_2 = .3$ and $\gamma = .5$. It is easy to check that $Q$ has two minima when $\gamma > \Gamma_1= .262...$, that is, for a sufficiently strong repellent charge.}
    \label{fig:example1}
\end{figure}
\begin{lemma}\label{lem:minima}
There exists $\Gamma_0 \in (0,1)$ such that for any $\Gamma_0<\gamma<1$, the external field $Q(x)$ given by \eqref{twochargest}-\eqref{z1z2} has two real minima. One of the real minima belongs to the interval $(-\infty,-1)$,
while the other one is in $(1,+\infty)$.
\end{lemma}
\begin{proof}
It is clear that $Q(x)$ is strictly increasing when $x\in[-1,1]$ so has no extremum  there. When $\gamma<1$, the derivative of $Q(x)$ has three (possibly complex) roots, while when $\gamma=1$, the derivative has two roots, those of
$$
-2x^{2}-(\beta_{1}^{2}-\beta_{2}^{2})x+(2+\beta_{1}^{2}+\beta_{2}^{2}),
$$
which clearly are real and of different signs. By continuity, it implies, that for $\gamma$ less than 1 and close to it, the derivative of $Q(x)$ has three real roots (one of them tending to infinity as $\gamma$ tends to 1). The smallest root
necessarily corresponds to a minimum of $Q(x)$, and belongs to the interval $(-\infty,-1)$. Then, the second one is a maximum, and the third one
is another minimum, both belonging to $(1,+\infty)$, and the latter tending to infinity as $\gamma$ tends to 1.
\end{proof}

\begin{remark}\label{rem:phasetransit}
For admissible values $t< T$, the fact that the external field \eqref{twochargest} has two minima was proved in \cite{MOR2015} and \cite{OrSL2015}--\cite{OrSL2016}
to be a sufficient condition for the existence of a {\it two-cut phase}, during which the support of the equilibrium measure consists of two disjoint intervals
(for polynomial and rational external fields, all charges being attractors).
In the present case, it could be checked from our subsequent results that it is not true anymore.
In those previous works it was also proved that, under certain conditions, there exists a two--cut phase even if there is a single minimum. The limit case takes place when a so-called ``type III singularity'' occurs  (see the above mentioned papers and \cite{KuML}). In principle, it seems natural to expect the same situation here, namely the existence of another value $0< \widetilde{\Gamma}_0 < \Gamma_0 < 1$ for which $Q$ has a type III singularity, for some admissible value $t< T=1-\gamma$, such that for $\gamma > \widetilde{\Gamma}_0$, there exists a two--cut phase. However, we will not study further this issue, since our main concern here is the weakly admissible case $t=T$.
\end{remark}
\subsection{Equilibrium measures on the real line}
In this section, we recall a few results about equilibrium measures on the real line in an admissible external field $Q$ which is assumed to be real-analytic, and in particular of the form (\ref{twochargest}) (and thus, in this case, it is assumed that the mass $t$ of the equilibrium measure satisfies $t<T$). We denote by
$$C_{\mu}(z)=\int\frac{d\mu(s)}{s-z},
$$
the Cauchy transform of a positive measure $\mu$ supported on a subset of $\C$.
\begin{theorem}[{\cite[Theorem 2]{MOR2015}}]
Assume there is a domain $\Omega$ in $\C$, containing $\R$, such that the admissible external field
$Q(x)$, $x\in\R$, is the trace of a (necessarily real) analytic function, still denoted $Q$, in $\Omega$. Then, there exists an analytic function $R$ in $\Omega$, real-valued on $\R$, such that the equilibrium measure $\mu$ with support $S\subset\R$ in the admissible external field $Q(x)$ satisfies
\begin{equation}\label{eq-MOR}
(C_{\mu}(z)+Q'(z))^{2}=R(z),\quad z\in\Omega\setminus S.
\end{equation}
\end{theorem}
\begin{remark}
If $Q$ is of the form \eqref{twochargest}-\eqref{z1z2} with the repellent charge on $\R$ (i.e. $\beta_2 = 0$), we have $Q(1)= +\infty$ and then
$1\notin\supp\mu_{t}$, see \cite[Remark I.1.4]{Saff:97}. It may be checked that the result in Theorem 3.3 is still applicable in that case, replacing $\R$ with a closed subset that excludes $1$.

\end{remark}
Throughout in the sequel, when we consider an equilibrium measure $\mu_t$ of a given mass $t$, we will denote its support by
$$S_{t}:=\supp(\mu_{t}).$$

In case $Q$ is of the form (\ref{twochargest}) and $t<T=1-\gamma$, following similar arguments as in \cite{OrSL2015} and \cite{OrSL2016}, it can be checked that the function $R$ is rational, and more precisely, after taking square root in (\ref{eq-MOR}), one may prove that
\begin{equation}\label{algebequat}
C_{\mu_{t}}(z)+Q'(z)=(T-t)\frac{\sqrt{A_{t}(z)}B_{t}(z)}{D(z)},\quad z\in\C\setminus (S_{t}\cup\{z_{1},\bar z_{1},z_{2},\bar z_{2}\}),
\end{equation}
with 
$$D(z)=(z-z_1)(z-\overline{z}_1)(z-z_2)(z-\overline{z}_2).$$
Here, $A_{t}$ is a monic polynomial of even degree, whose simple roots are the endpoints of $S_t$,
and the square root of $A_{t}$ is chosen so that it is analytic in $\C\setminus S_{t}$ and positive on the right of $S_{t}$. Finally
$B_{t}$ is a monic polynomial, and $\deg A_{t} B_{t}^2=6$. Since $A_{t}$ is at least of degree 2, $B_{t}$ is at most of degree 2. Both polynomials $A_{t}$ and $B_{t}$ are real and depend on $t$. In the sequel, their roots are respectively denoted by $a_1\leq a_{2}\leq\ldots\leq a_{n_{A}}$, and $b_j$, $j=1,\ldots,n_{B}$, where
$n_{A}:=\deg A_{t}$, $n_{B}:=\deg B_{t}$.

In addition, \eqref{algebequat} provides, by integrating and taking real parts, the following expression for the potential:
\begin{equation}\label{represpot}
V^{\mu_t} (z) + \Re Q(z) -c_{t}= (T-t)\Re \int_{a}^z \frac{\sqrt{A_{t}(x)}B_{t}(x)}{D(x)}dx,
\end{equation}
where $a$ is any point in the support $S_{t}$.
In view of (\ref{Frost1})-(\ref{Frost2}), when integrating in a gap of $S_{t}$, we have that the integral
$$\int_{a_{2j}}^{z}\frac{\sqrt{A_{t}(x)}B_{t}(x)}{D(x)}dx$$
is non-negative when $z\in(a_{2j},a_{2j+1})$ and vanishes when $z=a_{2j+1}$, it
implies that each gap of the support contains an odd number of zeros of $B_{t}$.
Moreover, inequality (\ref{Frost1}) implies that on the right of $a_{n_{A}}$, the largest root of $A_{t}$, there is an even number of zeros of $B_{t}$. The same holds true on the left of the smallest root $a_{1}$ (note that the integral in (\ref{represpot}) diverges as $z$ tends to
$\pm\infty$).

Equation \eqref{algebequat}, together with Plemelj formula (see e.g. \cite{Mushkelishvili}), or Stieltjes inversion formula,
\begin{equation*}
\mu'_{t}(x)=\frac{1}{\pi}\Im (C_{\mu_{t}})_{+}(x),\qquad x\in S_{t},
\end{equation*}
also provide the density of the equilibrium measure, namely,
\begin{equation*}
\mu'_t (x) = \frac{T-t}{\pi}\frac{\sqrt{|A_{t}(x)|}|B_{t}(x)|}{D(x)},\quad x\in S_t.
\end{equation*}
As in \cite{MOR2015} and \cite{OrSL2015}--\cite{OrSL2016}, we proceed
by studying the evolution of the equilibrium measure $\mu_t$ in the external field $Q$ in \eqref{twochargest} as the mass $t\in(0,T)$ approaches the limit value $T$.
In the next theorem, we show that the measures $\mu_{t}$ have a weak-* limit as $t$ tends to $T$, namely the equilibrium measure $\mu_{T}$, solution of the weakly admissible equilibrium problem.
\begin{theorem}\label{thm:waT}
One has:\\
i) The map
$t\in(0,T)\mapsto\mu_{t}$
is increasing and continuous with respect to the weak-* topology on the set of positive measures supported on the real axis.\\
ii) As $t$ tends to $T$, the measure
$\mu_{t}$ tends weak-* to a limit measure $\widetilde\mu_{T}$ of mass $T$ and support $\tilde S_{T}\subset\R$.\\
iii) The Cauchy transform of $\tilde\mu_{T}$ satisfies
\begin{equation}\label{algebequat2}
C_{\tilde\mu_{T}}(z)+Q'(z)=c\frac{\sqrt{A(z)}B(z)}{D(z)},\quad z\in\C\setminus (\tilde S_{T}\cup\{z_{1},\bar z_{1},z_{2},\bar z_{2}\}),
\end{equation}
where $c$ is some possibly complex constant, $A$ is a monic polynomial, whose simple roots are the (finite) endpoints of $\tilde S_T$,
$B$ is a monic polynomial of degree less than or equal to $2$, and $\deg A B^2<6$. Moreover, each gap of $\tilde S_{T}$ contains exactly one root of $B$.\\
iv) The measure $\tilde\mu_{T}$ has a density, given by
\begin{equation}\label{densityeqmeas2}
\tilde\mu'_T (x) = \frac{|c|}{\pi}\frac{\sqrt{|A(x)|}|B(x)|}{D(x)},\quad x\in \tilde S_T.
\end{equation}
v) The measure $\widetilde\mu_{T}$ integrates the logarithm at infinity, that is
\begin{equation}\label{log-inf}
\int\log(1+|x|)d\widetilde\mu_{T}(x)<\infty.
\end{equation}
vi) The measure $\tilde\mu_{T}$ has finite energy $I(\tilde\mu_{T})$.
\\
vii) As $t$ tends to $T$, the equilibrium constant $c_{t}$, see \eqref{eq-ct}, tends to a finite limit $c_{T}$.
\\
viii) The limit measure $\widetilde\mu_{T}$ coincides with the equilibrium measure $\mu_{T}$ of mass $T$ corresponding to the weakly admissible case.
\end{theorem}
\begin{proof}
Assertion i) is known in the case of an external field satisfying the condition at infinity
$\lim_{|x|\to\infty}Q(x)/\log|x|=\infty$, see \cite[Theorem 2]{BR:99}. It may be checked that the assertion still holds true in the admissible case \eqref{condadmiss},
with $t\in(0,T)$.
To derive assertion ii), we first prove that the family of measures $\mu_{t}$, $t\in(0,T)$, is tight, that is,
$$\forall\epsilon>0,\quad\exists~\text{compact set }K\subset\R,\quad
\forall t\in(0,T),\quad\mu_{t}(\R\setminus K)\leq\epsilon.$$
Indeed, choosing $K=S_{T-\epsilon}$, it is clear that $\mu_{t}(\R\setminus K)=0$ for $t\leq T-\epsilon$ and $\mu_{t}(\R\setminus K)<\epsilon$ for $t> T-\epsilon$. Thus, by Prohorov theorem (cf. \cite[Theorem 9.3.3]{DUD}), there is a subsequence $(\mu_{t_{n}})_{n}$ of $(\mu_{t})_{t\in(0,T)}$ weakly convergent to a measure $\tilde\mu_{T}$ of mass $T$. Since $\mu_{t}$ is increasing with $t$, the whole sequence $\mu_{t}$ actually converges to $\tilde\mu_{T}$.

Equation \eqref{algebequat2} in assertion iii) is obtained as the limit of \eqref{algebequat} as $t$ tends to $T$. Indeed, outside of $\tilde S_{T}$, $C_{\mu_{t}}$ tends pointwise to $C_{\tilde\mu_{T}}$, by the weak-* convergence of $\mu_{t}$ to $\tilde\mu_{T}$. On the right-hand side, since the multiplicative constant $T-t$ tends to zero, there is at least one root of $A_{t}B_{t}$ that tends to infinity, so that $\deg AB^{2}<6$. Also $\deg B\leq2$ since $\deg B_{t}\leq2$ for $t\in(0,T)$. The constant $c$ comes from the roots of $A_{t}B_{t}$ that tend to infinity as $t$ tends to $T$. For instance, one checks that, if the largest root of $A_{t}$ tends to $+\infty$ (so that $\tilde S_{T}$ is unbounded on the right) then $c\in i\R_{+}$ while if
the smallest root of $A_{t}$ tends to $-\infty$ then $c\in\R_{+}$. If both roots, on the left and on the right, tend to infinity, then $c\in i\R_{+}$.

The expression \eqref{densityeqmeas2} for the density of $\tilde\mu_{T}$ follows from \eqref{algebequat2} and Plemelj formula.
Assertion v) clearly follows from the expression \eqref{densityeqmeas2} for the density of $\tilde\mu_{T}$.
Finally, we verify that $I(\tilde\mu_{T})$ is finite. Since the density of $\tilde\mu_{T}$ is continuous, its potential $V^{\tilde\mu_{T}}$ is continuous as well. We may write
\begin{eqnarray*}
|V^{\tilde\mu_{T}}(x)| & \leq & \int|\log|x-y||d\tilde\mu_{T}(y)\\
& = & \int_{|x-y|\geq1}\log|x-y|d\tilde\mu_{T}(y)-\int_{y=x-1}^{x+1}\log|x-y|d\tilde\mu_{T}(y)\\
& \leq & T\log(1+|x|)+\int\log(1+|y|)d\tilde\mu_{T}(y)+\int_{y=-1}^{1}|\log|y||d\tilde\mu_{T}(y+x),
\end{eqnarray*}
where we have used that $|x-y|\leq(1+|x|)(1+|y|)$ in the last inequality. For $|x|$ large, the last integral is of order
$$\int_{y=-1}^{1}\frac{|\log|y||}{(y+x)^{3/2}}dy\leq\frac{2}{(|x|-1)^{3/2}}.$$
From the above follows in particular that, for $|x|$ large, the order of growth of the potential $V^{\tilde\mu_{T}}(x)$ is essentially bounded by $T\log(|x|)$. Together with the expression of the density of $\tilde\mu_{T}$, it implies that the energy
$$I(\tilde\mu_{T})=\int V^{\tilde\mu_{T}}(x)d\tilde\mu
_{T}(x)$$
is finite.

For assertion vii), recall that, for $t\in(0,T)$,
$$c_{t}=t^{-1}\left(I(\mu_{t})+\int Qd\mu_{t}\right).$$
In view of (\ref{log-inf}) and the expression (\ref{twochargest}) for $Q$, the integral $\int Q(x)d\tilde\mu_{T}(x)$ converges. Thus, for any $\epsilon>0$, there exists a compact set $K\subset\R$ such that
$$0\leq\int_{\R\setminus K} Q(x)d\tilde\mu_{T}(x)\leq\epsilon.$$
Since $\mu_{t}$ is increasing to $\tilde\mu_{T}$, the same inequalities hold true for all $\mu_{t}$, $t\in(0,T)$.
Together with the weak-* convergence of $\mu_{t}$ to $\tilde\mu_{T}$, it implies that
$$\int Qd\mu_{t}\to\int Qd\tilde\mu_{T}\quad\text{as}\quad t\to T.$$
Denoting by $\log^{+}$ and $\log^{-}$ the positive and negative parts of the $\log$ function, we have, as $t\to T$,
\begin{eqnarray*}
\mu_{t}'(z)\mu_{t}'(y)\log^{+}|z-y|\uparrow\tilde\mu_{T}'(z)\tilde\mu_{T}'(y)\log^{+}|z-y|,\\[10pt]
\mu_{t}'(z)\mu_{t}'(y)\log^{-}|z-y|\uparrow\tilde\mu_{T}'(z)\tilde\mu_{T}'(y)\log^{-}|z-y|,
\end{eqnarray*}
almost everywhere, with respect to the product Lebesgue measure, on $\tilde S_{T}\times\tilde S_{T}$. Here, we notice that the density $\mu_{t}'$ tends to $\tilde\mu_{T}'$ in $L^{1}(\R)$, so a subsequence tends pointwise almost everywhere to $\tilde\mu_{T}'$, and since the sequence is increasing, actually the entire sequence tends pointwise a.e. to $\tilde\mu_{T}'$. By monotone convergence, we obtain that
$I(\mu_{t})$ tends to $I(\tilde\mu_{T})$, and that $c_{t}$ tends to
$$c_{T}=T^{-1}\left(I(\tilde\mu_{T})+\int Qd\mu_{T}\right).$$
It remains to show assertion viii), that is $\widetilde\mu_{T}=\mu_{T}$. For each $t\in(0,T)$, we know from \eqref{Frost1} that
$$
V^{\mu_t}(x)+Q(x)  \geq c_t,\quad x\in \R.
$$
We apply the principle of descent, see Theorem \ref{th-descent}, to the sequences of positive measures $\widetilde\mu_{T}-\mu_{t}$ as $t$ tends to $T$, where we remark that assumption \eqref{log-t} follows from \eqref{log-inf}. Together with assertion vii), this gives
\begin{equation}\label{Frost1-tilde}
V^{\widetilde\mu_T}(x)+Q(x)  \geq c_T,\quad x\in \R.
\end{equation}
For each $t\in(0,T)$, we also know, see \eqref{Frost2}, that
$$
V^{\mu_t}(x)+Q(x)  = c_t,\quad x\in S_{t}.
$$
Let $x\in \widetilde{S}_{T}$, the support of $\widetilde\mu_{T}$, and let $t_{n}$ tends to $T$. From the inclusion
$$\widetilde{S}_{T}\subset\bigcap_{N=1}^{\infty}\overline{\bigcup_{n=N}^{\infty}S_{t_{n}}},$$
see \cite[Corollary 4 p.9]{Lan}, there exists a sequence $x_{n}\in S_{t_{n}}$ tending to $x$. Applying the inequality \eqref{descent} to the sequences $\mu_{t_{n}}$,
we derive, together with assertion vii) and the continuity of $Q$, that
\begin{equation}\label{Frost2-tilde}
V^{\widetilde\mu_T}(x)+Q(x)  = c_T,\quad x\in \widetilde{S}_{T}.
\end{equation}
From \eqref{Frost1-tilde}, \eqref{Frost2-tilde} and Proposition \ref{char-Frost} follows that $\widetilde\mu_{T}=\mu_{T}$.
\end{proof}

From now on, we will only denote by $\mu_{T}$ the common measure $\tilde\mu_{T}=\mu_{T}$, and by $S_{T}$ its support.

\subsection{The signed equilibrium measure}
In this section we assume that $t=T=1-\gamma$.
Let us denote by $\eta_{T}$ the opposite to the balayage of the measure (\ref{def-pair}) on $\R$:
$$\eta_T = \Bal(\delta_{z_1},\R) - \gamma \Bal(\delta_{z_2},\R).$$
Note that $\eta_{T}$ is the {\it signed equilibrium measure} in the external field $Q$, see Section \ref{sgn-meas}.
Throughout, we will denote the support of the positive part $\eta_{T}^{+}$ of $\eta_{T}$ by
$$S_{T}^{+}=\supp \eta_T^+.$$
We know from Lemma \ref{lem:realbalay} that
\begin{equation}\label{signedwadens1}
\eta'_T (x) = \frac{1}{\pi}\left(\frac{\beta_1}{(x+1)^2+\beta_1^2}-\frac{\gamma\beta_2}{(x-1)^2+\beta_2^2}\right).
\end{equation}
Since by Proposition 4.11, we have that $S_T \subset S_{T}^{+}$, one may easily derive situations where the support of the equilibrium measure in this weakly admissible external field is distinct from the whole real axis, see Figure \ref{fig:compT} 
for such cases.
\begin{figure}[ht]
    \begin{center}
        \includegraphics[scale=0.4]{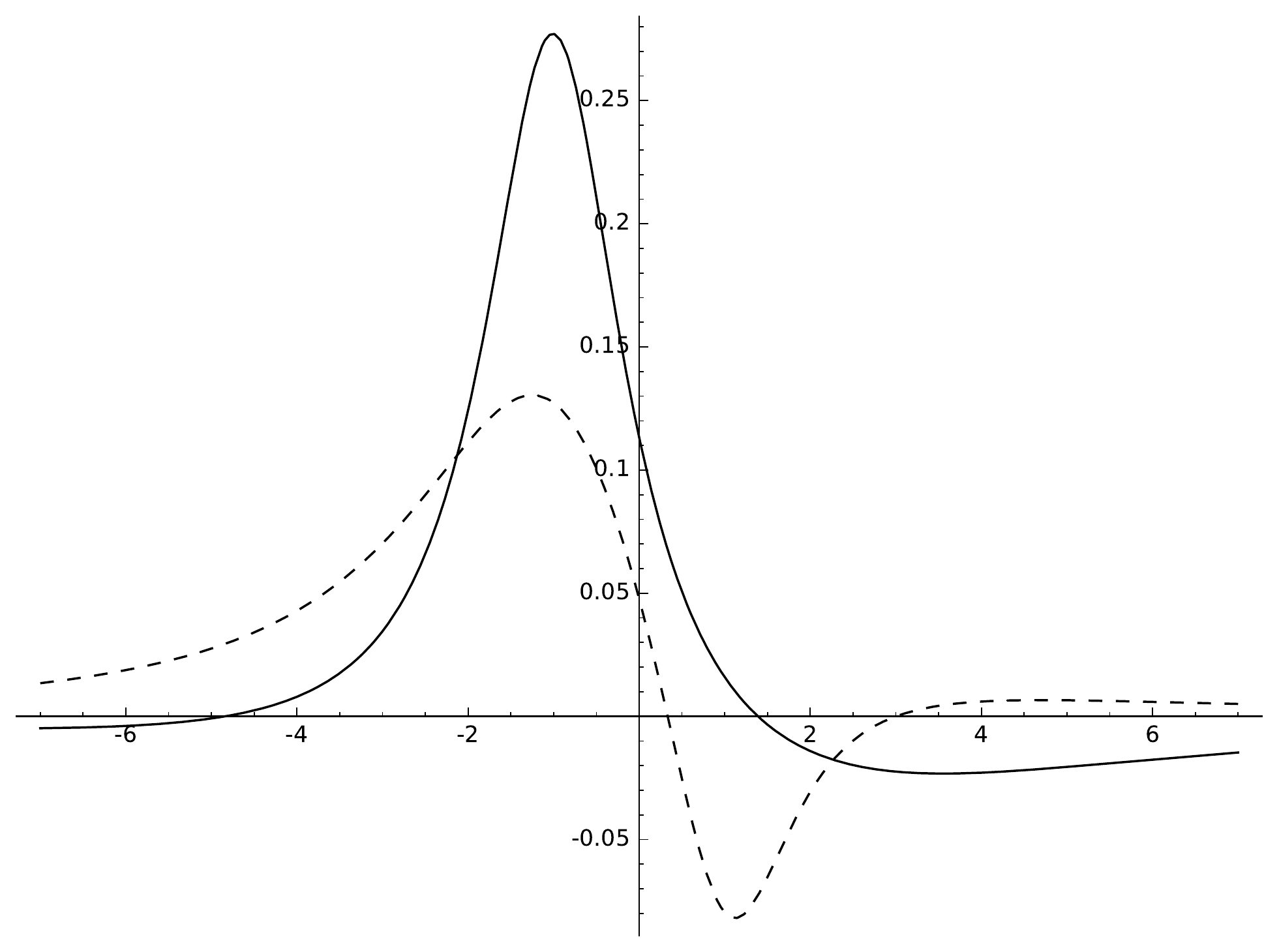}
    \end{center}
    \caption{A density (solid line) such that $S_T^{+}$ is compact ($\gamma = .75,\beta_1 = 1,\beta_2 = 5$), and a density (dashed line) such that $S_T^{+}$ is disconnected ($\gamma = 0.5,\beta_1 = 2,\beta_2 =1$)}
    \label{fig:compT}
\end{figure}

We first give a basic lemma about the signed equilibrium measure $\eta_{T}$.
\begin{lemma}\label{lem:threshold}
There exists a value $\Gamma_1 \in (0,1)$ such that $\eta_{T}$ is a positive measure for $\gamma \leq \Gamma_1$, and a signed measure (i.e. has a negative part in its Jordan decomposition) for $\gamma > \Gamma_1$.
\end{lemma}
\begin{proof}
It is clear that $\eta_{T}$ is positive when $\gamma$ is close to 0.
We seek a threshold value of $\gamma$ such that the density of the signed equilibrium measure $\eta'_T$ has a double real root. From \eqref{signedwadens1}, we must look for values of $\gamma$ for which the discriminant of the quadratic equation
\begin{equation*}
(\beta_1 - \gamma\beta_2)x^2-2(\beta_1 + \gamma\beta_2)x+\beta_1(1+\beta_2^2)-\gamma\beta_2(1+\beta_1^2)=0
\end{equation*}
vanishes. This yields:
$$\beta_1\beta_2(\beta_{1}\beta_{2}\gamma^{2}-(4+\beta_{1}^{2}+\beta_{2}^{2})\gamma+\beta_{1}\beta_{2})=0.$$
Since $\beta_1\beta_{2} \neq 0$, the third factor necessarily vanishes. As a polynomial in $\gamma$, it has
two positive real roots, one of them in $(0,1)$, namely
\begin{equation}\label{Gamma1}
\Gamma_{1}:=\frac{1}{2\beta_1\beta_2}
\left((\beta_1^2+\beta_2^2+4)-\sqrt{(\beta_1^2+\beta_2^2+4)^2-4\beta_{1}^{2}\beta_{2}^{2}}\right).
\end{equation}
It is easy to see that for $\gamma \in (0,\Gamma_1)$, the density $\eta'_T$ is positive on the whole real axis, while for $\gamma \in (\Gamma_1,1),$ it is negative in some subset of $\R$.
\end{proof}

When $\beta_1/\beta_2 < 1$, Theorem \ref{thm:compsupp} also shows the existence of a second threshold value, namely $\beta_1/\beta_2$,
such that the support of $\eta_{T}^{+}$ is unbounded (but disconnected) for $\Gamma_1<\gamma\leq\beta_1/\beta_2$ and bounded for $\beta_1/\beta_2<\gamma<1$.

We now describe the evolution of the support of the signed equilibrium measure (and, in turn, also derive information on the support of the equilibrium measure) as $\gamma$ moves from 0 to 1. At this point, let us introduce the semi-circle $\CC$, which contains $z_{1}$ and $z_{2}$ and whose center $x_{0}$ lies on the real axis, as depicted in Figure \ref{intro-circle}, and the associated points $x_{1},x_{2}$. It is easy to verify that
\begin{equation}\label{x0x1x2}
x_{0}=\frac14(\beta_{2}^{2}-\beta_{1}^{2}),\quad x_{1,2}=x_{0}\mp\frac14
\sqrt{(\beta_{2}^{2}-\beta_{1}^{2})^{2}+8(\beta_{1}^{2}+\beta_{2}^{2}+2)},
\end{equation}
where, in the second equation, $x_{1}$ (resp. $x_{2}$) corresponds to a $-$ (resp. $+$) sign.
The circle and the above points will be useful in our subsequent analysis.
\begin{figure}[ht]
\begin{center}
\def\svgwidth{10cm}
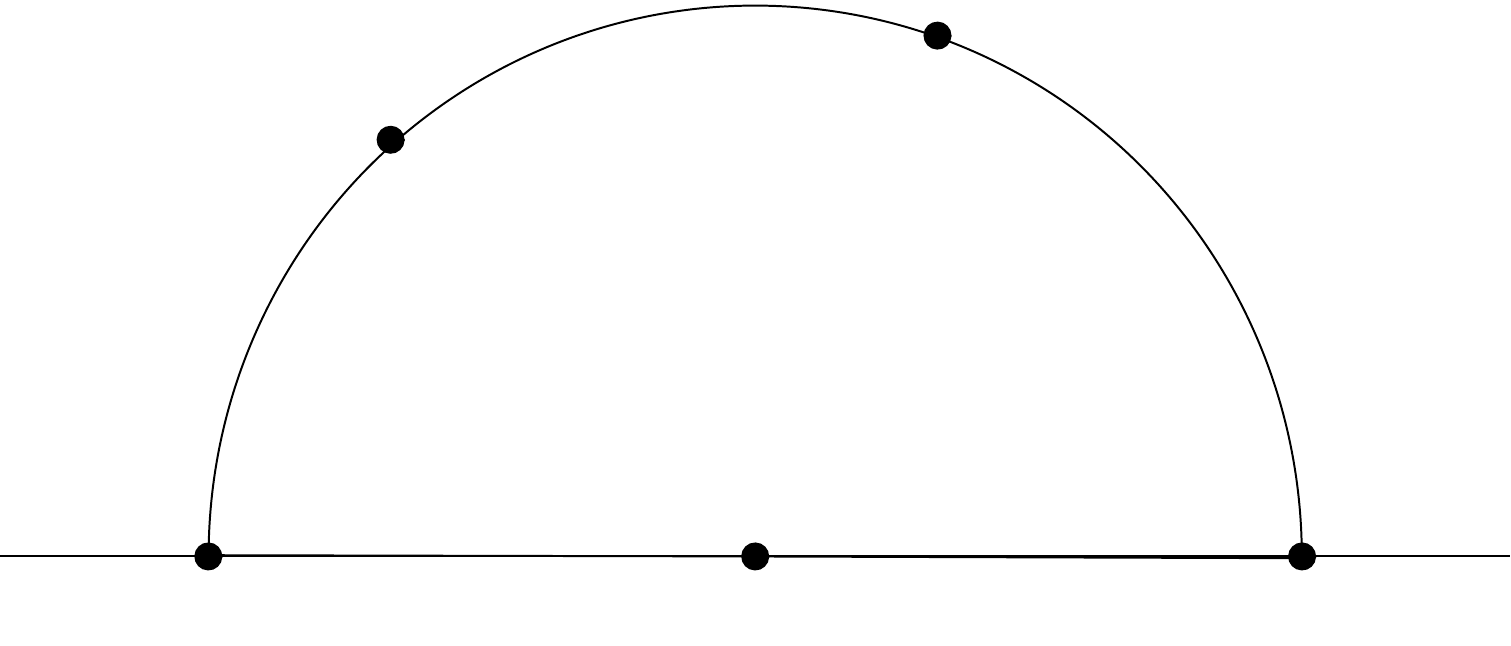
\caption{Semi-circle $\CC$.}
\label{intro-circle}
\end{center}
\end{figure}

\begin{proposition}\label{thm:phasessigned}
As $\gamma$ moves from 0 to 1, the following phases occur:

\begin{itemize}

\item {\textbf{Phase 1:}} $\gamma \in (0, \Gamma_1)$, with $\Gamma_1$ given by \eqref{Gamma1} : $\eta_T$ is a positive measure in $\R$. Thus,
$\mu_T = \eta_T$ and $S_T = \R$.

\item {\textbf{1st Transition:}} $\gamma = \Gamma_1$. The density of $\eta_T$ has a double real zero at $x_{2}$. One still has $\mu_T = \eta_T$ and $S_T = \R$.

\item {\textbf{Phase 2:}} $\gamma \in (\Gamma_1,\min(\beta_{1}/\beta_{2},1))$ : $\eta_T$ has a negative part supported in some compact interval $[A_1,A_2]$. Thus, $S_T \subset (-\infty, A_1] \cup [A_2, +\infty)$.

\end{itemize}

If $\beta_1 < \beta_2$,
we have the additional

\begin{itemize}

\item {\textbf{2nd Transition:}} $\gamma = \beta_{1}/\beta_{2}$. In this case, $S^+_T = (-\infty, A_1]$, implying that $S_T \subset (-\infty, A_1]$.

\item {\textbf{Phase 3:}} $\gamma \in (\beta_{1}/\beta_{2},1)$. Now, $S^+_T = [\widetilde{A}_2, A_1]$ and, therefore, $S_T \subset [\widetilde{A}_2, A_1]$.

\end{itemize}
\end{proposition}
\begin{proof}
When $\gamma\in[0,\Gamma_{1})$, we know by Lemma \ref{lem:threshold} that $\eta_{T}$ is a positive measure so it is clear that
$\eta_T=\mu_{T}$ and $S_{T}=\R$. During this phase, in the density (\ref{densityeqmeas2}) of $\mu_{T}$ (the fact that $\mu_{T}$ coincides with $\tilde\mu_{T}$ is shown in Theorem \ref{thm:waT}), the polynomial $A$ is constant, equal to 1, and $\deg B\leq2$. The polynomial $B$ cannot be constant (otherwise, as $\gamma$ goes from 0 to $\Gamma_{1}$, $\eta_{T}$ would be a multiple of a constant measure) nor be of degree 1, since $\eta_{T}'$ does not vanish on $\R$. So $B$ is of degree 2 with a pair of conjugate roots denoted by $b$ and $\bar b$. When $\gamma=0$, the measure $\eta_{T}$ is independent of $z_{2}$, so $B(z)$ should cancel the factors $(z-z_{2})(z-\bar z_{2})$ in $D(z)$, thus $b=z_{2}$. For $\gamma\in(0,\Gamma_{1})$, we know from (\ref{signedwadens1}) that $b$ and $\bar b$ are the conjugate roots of
\begin{equation}\label{conjroots}
\beta_{1}((x-1)^{2}+\beta_{2}^{2})-\gamma\beta_{2}((x+1)^{2}+\beta_{1}^{2}).
\end{equation}
It is then easy to show that the modulus $|b-x_{0}|$ is independent from $\gamma$, and equal to the radius of $\CC$. Hence, $b$ stays on $\CC$, starting from $z_{2}$ when $\gamma=0$, arriving at $x_{1}$ or $x_{2}$ when $\gamma=\Gamma_{1}$. Since \eqref{conjroots} implies that $\Re b = (b+\overline{b})/2$ is increasing with $\gamma$, the final location of $b = \overline{b}$ in this phase can only be $x_{2}$, which shows the assertion about the location of the double zero at the first transition.
The assertions about the subsequent phases easily follow from the expression (\ref{signedwadens1}) of $\eta_{T}$.
\end{proof}
\begin{remark}\label{rem:endpoints}
In the last item (Phase 3), we used the notation $\widetilde{A}_2$ for the left endpoint in order to exhibit the move of the initial endpoint $A_2$ when $\gamma$ grows from $\Gamma_1$ to $1$, namely: when $\gamma = \Gamma_1,$ we have that $A_2 = A_1$ equal the double real root of $\eta'_T$. Then, as $\gamma$ increases, $A_2$ increases and reaches $+\infty$ for $\gamma = \beta_{1}/\beta_{2}$. Then, for $\gamma > \beta_{1}/\beta_{2}$, $A_2$ comes in from $-\infty$ and becomes the new left endpoint $\widetilde{A}_2$.
\end{remark}
The above proof allows one to give an alternative expression for $\Gamma_{1}$. Indeed, taking the residues at $z_{1}$ and $z_{2}$ in (\ref{algebequat2}) when $\gamma=\Gamma_{1}$ gives the two equations
$$\frac12=\frac{c(z_{1}-x_{2})^{2}}{2\beta_{1}(z_{1}-z_{2})(z_{1}-\bar z_{2})},\qquad
-\frac{\Gamma_{1}}{2}=\frac{c(z_{2}-x_{2})^{2}}{2\beta_{2}(z_{2}-z_{1})(z_{2}-\bar z_{1})}.$$
Dividing the second equation by the first one, and taking modulus, leads to
\begin{equation}\label{G1bis}
\Gamma_{1}=\frac{\beta_{1}}{\beta_{2}}\frac{|z_{2}-x_{2}|^{2}}{|z_{1}-x_{2}|^{2}}.
\end{equation}

\begin{figure}[ht]
\centering
\includegraphics[scale=.5]{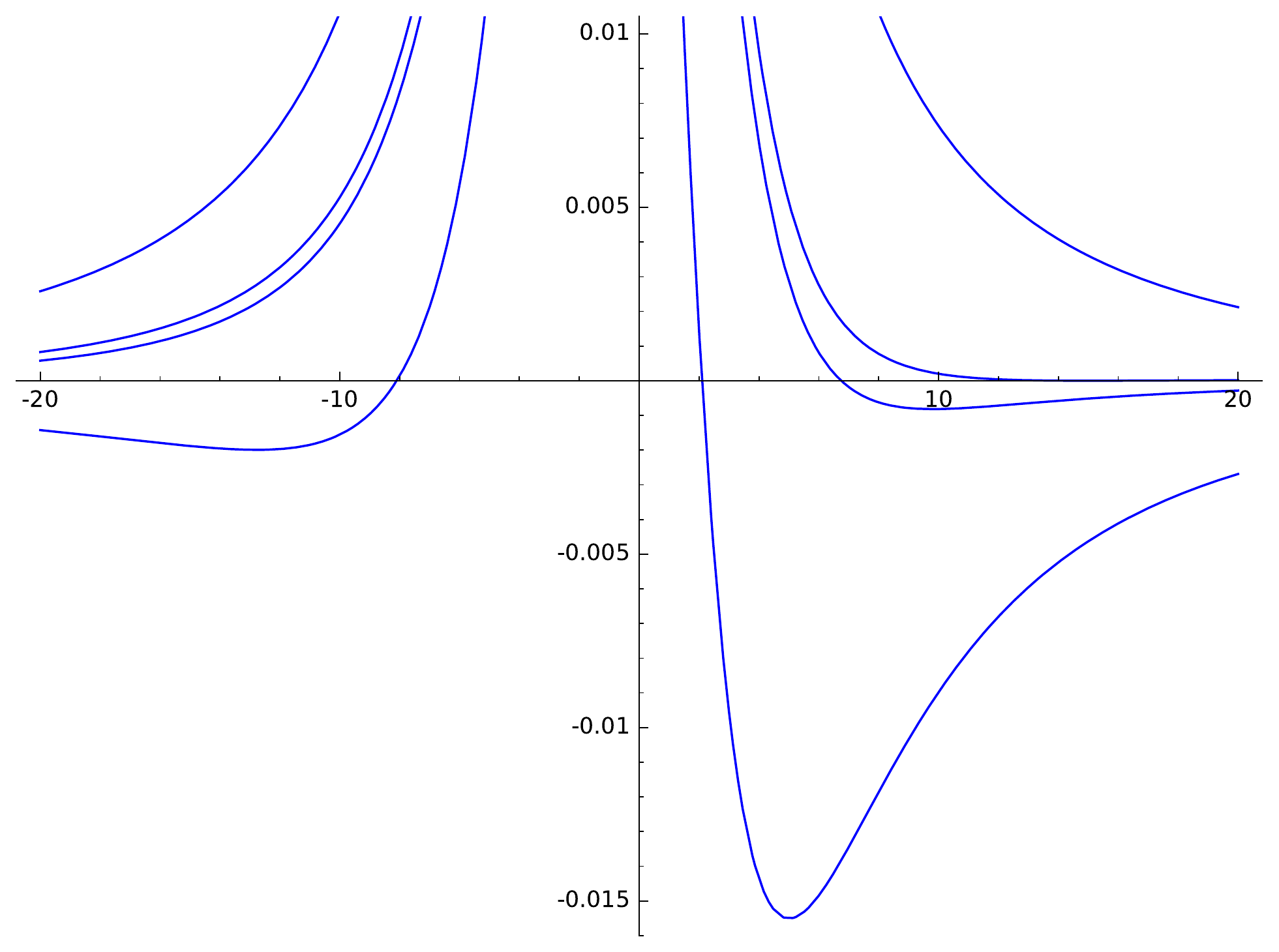}
\caption{Densities of $\eta_{T}(x)$ for $\gamma=0,0.43...,0.5,1$. Here $\beta_{1}=3$ and $\beta_{2}=6$. The central part where the density attains its maximum is not shown. As $\gamma$ increases, the total mass $T$ decreases and so does the sequence of densities. The value $\gamma=\Gamma_{1}=0.43...$ is the critical value where the density has a double root. For $x$ large in modulus, the density behaves respectively like $x^{-2}$, $-x^{-3}$, $-x^{-2}$ as $\gamma<0.5$, $\gamma=0.5$, $\gamma>0.5$ (up to positive multiplicative factors).}
\label{sign-equil}
\end{figure}

\subsection{The support of the equilibrium measure $\mu_{T}$}
The aim of this section is to describe the dynamics of the support $S_{T}$ of $\mu_{T}$ as $\gamma$ moves from 0 to 1.
As we have seen in Theorem \ref{thm:waT}, $\deg AB^{2}<6$ and each gap of $S_T$ must contain a root of $B$, thus only the following cases can occur:
\begin{itemize}
    \item [(i)] $\deg(A)=0$ (i.e. $S_T=\mathbb{R}$) and $\deg(B)\leq 2$,
    \item [(ii)] $\deg(A)=1$ (i.e. $S_T$ is a segment bounded by one side) and $\deg(B)\leq 2$,
    \item [(iii)] $\deg(A)=2$ (i.e. $S_T$ is a bounded interval or the complement of a bounded interval) and $\deg(B)\leq 1$,
    \item [(iv)] $\deg(A)=3$ (i.e. $S_T$ consists of two intervals, one bounded and the other one unbounded on one side) and $\deg(B)\leq 1$.
\end{itemize}
We begin with a lemma.
\begin{lemma}\label{lem-not-occur}
Case (iv) above cannot occur. In particular, $\deg AB\leq3$.
\end{lemma}
\begin{proof}
Equation (\ref{algebequat2}) implies that
\begin{equation}\label{algebequat3}
V^{\mu_T} (z) + \Re Q(z) -c_{T}= c\Re \int_{a}^z \frac{\sqrt{A(x)}B(x)}{D(x)}dx,
\end{equation}
where $a$ is any point in the support $S_{T}$. The integral on the right-hand side is a multivalued function of $z$, but its real part is well-defined in $\C\setminus\{z_{1},\bar z_{1},z_{2},\bar z_{2}\}$. Since its integrand behaves like $\OO(x^{-3/2})$ as $x$ tends to infinity, it is also well-defined at $\infty$ and vanishes there because of (\ref{algebequat3}) and \eqref{Frost1}--(\ref{Frost2}). In particular, one should have
\begin{equation*}
\int_{a_{3}}^\infty \frac{\sqrt{A(x)}B(x)}{D(x)}dx  =0,
\quad\text{if }  S_{T}=(-\infty,a_{1}]\cup[a_{2},a_{3}],
\end{equation*}
and
\begin{equation*}
\int_{a_{1}}^{-\infty} \frac{\sqrt{A(x)}B(x)}{D(x)}dx=0,
\quad\text{if }  S_{T}=[a_{1},a_{2}]\cup[a_{3},\infty),
\end{equation*}
which contradicts the fact that the (unique) root of $B$ lies between the two intervals of $S_{T}$.
\end{proof}
Since our main result concerns the dynamics with respect to the parameter $\gamma$, we compute in the next lemma the derivatives of the roots of the polynomials $A$ and $B$ in (\ref{algebequat2}) with respect to $\gamma$.
It will allow us to describe the evolution of the endpoints of the support,
as the strength $\gamma$ of the repellent varies. This follows the dynamical approach taken by Buyarov and Rakhmanov, in their seminal paper \cite{BR:99}, with respect to the mass $t$ of the equilibrium measure.
In previous papers \cite{MOR2015}, \cite{OrSL2015} and \cite{OrSL2016}, an analysis of the dynamics, with respect to the mass $t$ and other parameters of the external field $Q$, was performed, but just for admissible external fields.

In the sequel, we simply denote by $\dot f$ the derivative of a function $f$ with respect to~$\gamma$.
\begin{lemma}\label{dynamics}
Assume that 
the support $S_{T}$ has two finite endpoints $a_{1}$, $a_{2}$, and $B$ is of degree 1, so that
$$A(z)=(z-a_{1})(z-a_{2}),\qquad B(z)=(z-b),$$
with a root $b$ of $B$ outside of $S_{T}$. We denote by $h$ and $k$ the intersections of the internal and external bisectors of $\hat{a_{1}z_{2}a_{2}}$ with the real line such that the circle with diameter $(h,k)$ goes through $z_{2}$.

Then, the following holds,
\\
i) if $S_{T}=(-\infty,a_{1}]\cup[a_{2},\infty)$,
the derivatives of the roots $a_{1},a_{2}$ of $A$ and $b$ of $B$ with respect to $\gamma$ satisfy
\begin{align}
\dot a_{j}  & =\frac{2i\Im((a_{j}-\bar z_{2})\sqrt{A(z_{2})})|a_{j}-z_{1}|^{2}}{cA'(a_{j})(a_{j}-b)}, \quad j=1,2,\label{der-a}\\[10pt]
\dot b & =\frac{i\Im((b-\bar z_{2})\sqrt{A(z_{2})})|b-z_{1}|^{2}}{cA(b)}, \label{der-b}
\end{align}
where $A'(z)$ denote, as usual, the derivative of $A(z)$ with respect to $z$. The constant $c$ lies in $i\R_{+}$ and depends on $\gamma$.
Let $c=id$, $d>0$. Then, (\ref{der-a})--(\ref{der-b}) can be rewritten as
\begin{align}
\dot a_{j}  & =\frac{2\Im(\sqrt{A(z_{2})})(a_{j}-h)|a_{j}-z_{1}|^{2}}{dA'(a_{j})(a_{j}-b)}, \quad j=1,2,\label{der-a-bis}\\[10pt]
\dot b & =\frac{\Im(\sqrt{A(z_{2})})(b-h)|b-z_{1}|^{2}}{dA(b)}. \label{der-b-bis}
\end{align}
ii) if $S_{T}=[a_{1},a_{2}]$ and $b>a_{2}$,
the derivatives of the roots $a_{1},a_{2}$ of $A$ and $b$ of $B$ with respect to $\gamma$ satisfy
\begin{align}
\dot a_{j}  & =\frac{2\Re((a_{j}-\bar z_{2})\sqrt{A(z_{2})})|a_{j}-z_{1}|^{2}}{cA'(a_{j})(a_{j}-b)}, \quad j=1,2,\label{der-a1}\\[10pt]
\dot b & =\frac{\Re((b-\bar z_{2})\sqrt{A(z_{2})})|b-z_{1}|^{2}}{cA(b)}, \label{der-b1}
\end{align}
The constant $c$ is negative and depends on $\gamma$.
Let $c=-d$, $d>0$. 
Then, (\ref{der-a1})--(\ref{der-b1}) can be rewritten as
\begin{align}
\dot a_{j}  & =-\frac{2\Re(\sqrt{A(z_{2})})(a_{j}-k)|a_{j}-z_{1}|^{2}}{dA'(a_{j})(a_{j}-b)}, \quad j=1,2,\label{der-a1-bis}\\[10pt]
\dot b & =-\frac{\Re(\sqrt{A(z_{2})})(b-k)|b-z_{1}|^{2}}{dA(b)}. \label{der-b1-bis}
\end{align}
\end{lemma}
\begin{proof}
We first consider case i). Differentiating the right-hand side of (\ref{algebequat2}) with respect to $\gamma$, we get
$$\frac{\dot c A(z)B(z)+(c/2)\dot A(z)B(z)+cA(z)\dot B(z)}{\sqrt{A(z)}D(z)}.$$
Since the derivative of the left-hand side of (\ref{algebequat2}) has no pole at $z_{1}$ and $\bar z_{1}$, the factors $(z-z_{1})(z-\bar z_{1})$ should cancel in the above ratio. Moreover, by Lemma \ref{lem-not-occur}, the degree of the numerator is at most 3. Thus, the previous ratio can be rewritten as
$$\frac{mz+n}{\sqrt{A(z)}(z-z_{2})(z-\bar z_{2})},$$
where $m$ and $n$ can be computed from the residues of the left-hand side of (\ref{algebequat2}) at $z_{2}$ and $\bar z_{2}$. Namely, we have
\begin{equation}\label{eq-mn}
\frac{mz_{2}+n}{\sqrt{A(z_{2})}2i\beta_{2}}=-\frac{1}{2}, \qquad
\frac{m\bar z_{2}+n}{\sqrt{A(\bar z_{2})}2i\beta_{2}}=\frac{1}{2}.
\end{equation}
Together with the identity $\sqrt{A(\bar z_{2})}=-\bar{\sqrt{A(z_{2})}}$ which follows from the choice of branch cuts, this leads to
$$m=-i\Im\sqrt{A(z_{2})},\qquad n=i\Im(\bar z_{2}\sqrt{A(z_{2})}).$$
Choosing $z$ as a root of $A(z)$ or $B(z)$ in the identity
$$\dot c A(z)B(z)+(c/2)\dot A(z)B(z)+cA(z)\dot B(z)=
(mz+n)(z-z_{1})(z-\bar z_{1}),$$
gives (\ref{der-a}) and (\ref{der-b}). The fact that $c\in i\R_{+}$ follows from (\ref{algebequat2}) and Plemelj formula, from which we know that $\Im(c(\sqrt{A(x)})_{+}B(x))\geq0$ for $x\in S_{T}$, together with the facts that
$(\sqrt{A(x)})_{+}<0$ and $B(x)<0$ on $(-\infty,a_{1}]$ (or $(\sqrt{A(x)})_{+}>0$ and $B(x)>0$ on $[a_{2},\infty)$).
Finally, $mz+n=m(z-h)$ since, for $z$ real,
$\Im((\bar z_{2}-z)\sqrt{A(z_{2})})$ vanishes precisely when $z=h$. This implies
 (\ref{der-a-bis}) and (\ref{der-b-bis}).

 Let us now consider case ii). The beginning of the proof is identical to case i) up to equations (\ref{eq-mn}). Now we have $\sqrt{A(\bar z_{2})}=\bar{\sqrt{A(z_{2})}}$, and this leads to
$$m=-\Re\sqrt{A(z_{2})},\qquad n=\Re(\bar z_{2}\sqrt{A(z_{2})}).$$
The fact that $c<0$ follows from (\ref{algebequat2}) and Plemelj formula, from which we know that $\Im(c(\sqrt{A(x)})_{+}B(x))\geq0$ for $x\in S_{T}$, together with the facts that
$(\sqrt{A(x)})_{+}\in i\R_{+}$ and $B(x)<0$ on $[a_{1},a_{2}]$. Finally,  (\ref{der-a1-bis}) and (\ref{der-b1-bis}) follows from the fact that $mz+n=m(z-k)$ since it is easy to verify that $\Re((\bar z_{2}-z)\sqrt{A(z_{2})})$ vanishes when $z=k$.
\end{proof}

Equipped with the previous lemma, we are in a position to state our main result in this section. It describes the dynamics of the equilibrium measure $\mu_{T}$ in the weakly admissible external field $Q$ as the parameter $\gamma$ grows from 0 to 1.
In different situations, such as those studied in the papers \cite{MOR2015,OrSL2015,OrSL2016}, the measure $\mu_{T}$, limit of the equilibrium measures $\mu_{t}$ when $t\rightarrow T$, always had its support equal to the whole real axis.
 As the following theorem shows,
the behavior of $\mu_{T}$ in the present situation can be different.

Let us first define the constant $\Gamma_{2}$ by
\begin{equation}\label{Gamma2}
\Gamma_{2}:=\frac{\beta_{1}}{\beta_{2}}\frac{|z_{2}-x_{2}|}{|z_{1}-x_{2}|}<1,
\end{equation}
and recall that $\Gamma_{1}$ was defined in (\ref{Gamma1}), see also (\ref{G1bis}). Note that $\Gamma_{1}<\Gamma_{2}$.
\begin{theorem}\label{thm:main}
As $\gamma$ grows from $0$ to $1$, the following phase diagram takes place:
\begin{itemize}
\item \textbf{Phase 1}: For $\gamma\in[0,\Gamma_1)$, we have $S_T=\mathbb{R}$, and $\mu_{T}$ equals $\eta_{T}$, given by (\ref{signedwadens1}), or alternatively,
\begin{equation}\label{dens-1}
\mu_{T}(x)=\frac{d}{\pi}\frac{ |x-b|^{2}}{D(x)}dx,
\end{equation}
with some $d>0$. As $\gamma$ moves from 0 to $\Gamma_{1}$, the point $b$ describes the arc from $z_{2}$ to $x_{2}$ on the circle $\CC$.
\item \textbf{1st transition}: For $\gamma=\Gamma_1$, $b$ in (\ref{dens-1}) equals $x_{2}$ and the density of $\mu_{T}$ has a double zero at this point. The point $b$ will stay at $x_{2}$ in the subsequent phases.
\item \textbf{Phase 2}: For $\gamma\in(\Gamma_1,\Gamma_2)$, the support $S_T$ is the union of two semi-infinite intervals $(-\infty,a_1]\cup[a_2,+\infty)$. The two endpoints $a_{1}$ and $a_{2}$ coincide with $x_{2}$ when $\gamma=\Gamma_{1}$. As $\gamma$ grows from $\Gamma_{1}$ to $\Gamma_{2}$, $a_{1}$ moves to the left and $a_{2}$ to the right of $x_{2}$ while the bisector of $\hat{a_1z_2a_2}$ constantly passes through $x_2$.
The equilibrium measure is given on $S_{T}$ by
$$\mu_{T}(x)=\frac{d}{\pi} \frac{\sqrt{(x-a_1)(a_2-x)}|x-x_2|}{D(x)}dx,$$
where $d>0$.
\item \textbf{2nd transition}: At $\gamma=\Gamma_2$, $a_{1}$ reaches the center $x_{0}$ of the circle $\CC$ while $a_{2}$ reaches $+\infty$. The support $S_T$ equals $(-\infty,x_0]$, and the equilibrium measure is given by
$$\mu_{T}(x)=\frac{d}{\pi}\frac{\sqrt{(x_0-x)}(x_2-x)}{D(x)}dx,$$
where $d>0$.
\item \textbf{Phase 3}: For $\gamma\in(\Gamma_2,1)$, the support $S_T$ is a finite segment $[a_1,a_2]$, where $a_2$ is the continuation of the previous $a_1$ and $a_1$ is the continuation (through $\infty$) of the previous $a_2$. The equilibrium measure is given by
$$\mu_{T}(x)=\frac{d}{\pi} \frac{\sqrt{(x-a_1)(a_2-x)}(x_2-x)}{D(x)}dx,$$
where $d>0$.
As $\gamma$ grows from $\Gamma_{2}$, the point $a_1$ increases from $-\infty$, while $a_2$ decreases from $x_{0}$. They collide at $x_{1}$ when $\gamma= 1$. Furthermore, the bisector of $\hat{a_1z_2a_2}$ constantly passes through $x_1$. 
\end{itemize}
\end{theorem}
\begin{figure}[ht]
\begin{center}
\def\svgwidth{10cm}
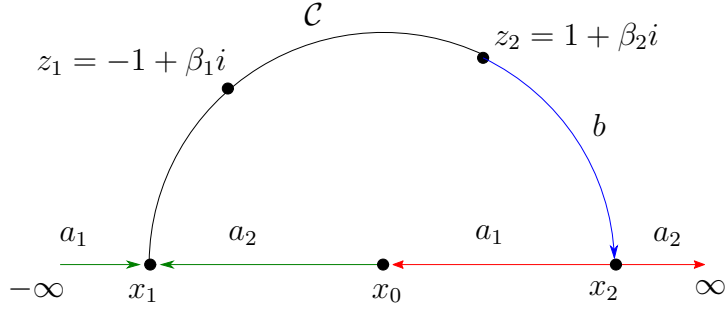
\caption{Dynamics of the finite endpoints $a_{1}<a_{2}$ of the support $S_{T}$. Phase 1 (blue): $S_{T}$ is the whole real line, the root $b$ goes along the arc of circle from $z_{2}$ to $x_{2}$ and then stay at this point for the two subsequent phases. Phase 2 (red): $S_{T}$ has a gap which appears as a point at $x_{2}$, increasing to the semi-infinite interval $(x_{0},\infty)$. Phase 3 (green): The previous gap continues to increase, or equivalently $S_{T}$ diminishes from $(-\infty,x_{0})$ to a single point at $x_{1}$.}
\end{center}
\end{figure}
\begin{remark}\label{rem:comparison}
Note that the number of phases is constant, namely three, for any  choice of $\beta_{1},\beta_{2}>0$.
Observe also that the threshold value $\Gamma_2$ is smaller than $\min(1,\beta_1/\beta_2)$.
Hence, while there does not always exist a phase where $S_{T}^{+}$
is bounded, such a phase always exists for $S_{T}$.
In case such a phase exists for the signed equilibrium measure, it starts after the one for the equilibrium measure.
Also the numbers of components of $S_T$
and $S^+_T$ do not always agree; indeed, as shown in Proposition \ref{thm:phasessigned} and Theorem \ref{thm:main}, for $\gamma \in (\Gamma_2, \min \{1, \beta_1/\beta_2\}),$ $S^+_T$ has two unbounded connected components while $S_T$ has just a bounded one.
\end{remark}
\begin{remark}\label{apoll}
During Phase 2, the fact that $x_{2}$ lies on the bisector of $\hat{a_{1}z_{2}a_{2}}$ implies that the circle $\CC$ is an Apollonius circle for $a_{1}$ and $a_{2}$, that is, there exists a $m>0$ such that, for any point $P\in\CC$, $d(P,a_{2})=md(P,a_{1})$. It is then easy to derive that the geometric mean of the distances $|a_{1}-x_{0}|$ and $|a_{2}-x_{0}|$ is the radius of circle. The same holds true during Phase 3.
\end{remark}
\begin{proof}[Proof of Theorem \ref{thm:main}]
When $\gamma\in[0,\Gamma_{1})$, we know that $\eta_{T}$ is a positive measure and
$\mu_T = \eta_T$, so the assertions about the first phase and the first transition have already been shown in the proof of Proposition \ref{thm:phasessigned}.

When $\gamma$ becomes larger than $\Gamma_1$, a negative part in the measure $\eta_{T}$ appears around $x_{2}$, see Lemma \ref{lem:threshold}.
Since the support $S_{T}$ of $\mu_{T}$ is a subset of the support of $\eta_{T}^{+}$,
we deduce that $S_{T}$ splits into two semi-infinite intervals $(-\infty,a_{1}]\cup[a_{2},\infty)$.
Namely, the two conjugate roots of $B$ collide at $x_{2}\in\R$ when $\gamma=\Gamma_1$ and then they split into two roots, $a_{1}$ and $a_{2}$ of $A$, and a root $b$ of $B$.

For $\gamma>\Gamma_{1}$, we see from assertion i) in Lemma \ref{dynamics}, and equations (\ref{der-a-bis}) and (\ref{der-b-bis}), that
$$\dot a_{1}<0,\quad\dot a_{2}>0,\quad\dot b(b-h)<0.$$
Hence, as $\gamma$ increases, $a_{1}$ is moving to the left, $a_{2}$ to the right, and $h$ is an attractor point for $b$. On the other hand, taking the residues at $z_{1}$ and $z_{2}$ in (\ref{algebequat2}) gives the two equations
$$\frac12=\frac{c\sqrt{A(z_{1})}B(z_{1})}{2\beta_{1}(z_{1}-z_{2})(z_{1}-\bar z_{2})},\qquad
-\frac{\gamma}{2}=\frac{c\sqrt{A(z_{2})}B(z_{2})}{2\beta_{2}(z_{2}-z_{1})(z_{2}-\bar z_{1})},$$
or equivalently
\begin{equation}\label{residues}
\beta_{1}(z_{1}-z_{2})(z_{1}-\bar z_{2})=c\sqrt{A(z_{1})}B(z_{1}),\quad
-\beta_{2}\gamma(z_{2}-z_{1})(z_{2}-\bar z_{1})=c\sqrt{A(z_{2})}B(z_{2}).
\end{equation}
Note that taking the residues at $\bar z_{1}$ and $\bar z_{2}$ in (\ref{algebequat2}) leads to the same equations as above where we apply conjugation on both sides.
The left-hand side in the second equation in (\ref{residues}) has a constant argument. Hence, also 
$$\frac12\arg \left((z_{2}-a_{1})(z_{2}-a_{2})\right)+\arg(z_{2}-b)
=\arg(z_{2}-h)+\arg(z_{2}-b)$$
remains constant as $\gamma$ increases. Assume $h$ moves to the right of $x_{2}$. Then, $b$ moves to the right of $x_{2}$, which contradicts the fact that the previous sum of arguments remains constant. The same argument holds if we assume that $h$ moves to the left of $x_{2}$.
It follows that both points remain fixed, at the location $x_{2}$ where they coincide  when $\gamma=\Gamma_{1}$.

The next question is to determine if one of the endpoints $a_{1}$ and $a_{2}$ reaches infinity before $\gamma=1$. Since $z_{2}$ is on the left of $x_{2}$ and $x_{2}$ lies on the bisector of the angle $\hat{a_{1}z_{2}a_{2}}$, the first endpoint to possibly reach infinity can only be $a_{2}$.
Assume there is a $\gamma=:\Gamma_{2}\in(0,1)$ for which $a_{2}=+\infty$. Then
$a_{1}=x_{0}$ since $\hat{x_{0}z_{2}x_{2}}=\hat{x_{2}z_{2}\infty}$. Moreover, we deduce from dividing the second equation by the first one in (\ref{residues}) and taking modulus that
$$\Gamma_{2}=\frac{\beta_{1}}{\beta_{2}}\sqrt{\frac{|z_{2}-a_{1}|}{|z_{1}-a_{1}|}}\frac{|z_{2}-x_{2}|}{|z_{1}-x_{2}|}
=\frac{\beta_{1}}{\beta_{2}}\frac{|z_{2}-x_{2}|}{|z_{1}-x_{2}|}
=\frac{\sin(\theta_{1})}{\sin(\theta_{2})}<1,$$
where $\theta_{i}=\arg(z_{i}-x_{2})$, $i=1,2$. Since we have found that $\Gamma_{2}<1$ we may thus conclude that indeed the endpoint $a_{2}$ reaches $+\infty$ before $\gamma=1$.

For the study of phase 3, we find it easier to let $\gamma$ decrease from the value 1. For the value $\gamma=1$, the mass of $\mu_{T}$ vanishes and for $\gamma$ a little less than 1, $S_{T}$ is a small segment $[a_{1},a_{2}]$ around the minimum of $Q(x)=\log|x-z_{1}|-\log|x-z_{2}|$. This follows from the Frostman inequalities
(\ref{Frost1})--(\ref{Frost2}) that are satisfied in the weakly admissible case as well, see Section \ref{unbdd}. When $\gamma=1$,
$$Q(x)=\log\frac{|x-z_{1}|}{|x-z_{2}|},$$
and it is easy to verify that $Q(x)$ attains its minimum at $x_{1}$ and its maximum at $x_{2}$. Hence, $a_{1}=a_{2}=h=x_{1}$ and thus $k=x_{2}$ (recall that the point $k$ was defined in Lemma \ref{dynamics}). Also
$$Q'(x)=-\frac{4(x-x_{1})(x-x_{2})}{D(x)},$$
and in view of (\ref{algebequat2}), we may deduce, still for $\gamma=1$, that the constant $c$ in (\ref{algebequat2}) is negative and the polynomial $B$ is of degree 1, with a root $b$ equal to $x_{2}$. Next, for $\gamma$ less than 1, applying assertion ii) in Lemma \ref{dynamics} gives that
$$\dot a_{1}>0,\quad\dot a_{2}<0,\quad\dot b(b-k)<0,$$
so that, as $\gamma$ decreases from 1, $a_{1}$ goes to the left of $x_{1}$, $a_{2}$ goes to the right of $x_{1}$, and $b$ is attracted by $k$. Let us check that $b$ and $k$ remain at $x_{2}$ during that phase. This is equivalent to showing that $h$ remains at $x_{1}$. The second equation in (\ref{residues}), which is also valid in that phase, shows again that
$
\arg(z_{2}-h)+\arg(z_{2}-b)$
remains constant as $\gamma$ decreases. Assume $h$ moves to the right of $x_{1}$. Then, $k$ moves to the right of $x_{2}$ and so does $b$, which contradicts the fact that the previous sum of arguments remains constant. The same argument holds if we assume that $h$ moves to the left of $x_{1}$.

As $\gamma$ decreases, $a_{2}$ remains to the left of $b=x_{2}$, and thus remains bounded. We next study if the left endpoint $a_{1}$ can reach $-\infty$. Assume it happens at a particular time $\gamma=\tilde\gamma$. Then
$a_{2}=x_{0}$ since $\hat{x_{1}z_{2}x_{0}}=\hat{(-\infty)z_{2}x_{1}}$. As in the previous phase, if we divide the second equation by the first one in (\ref{residues}) and take modulus, we get
$$\tilde\gamma=\frac{\beta_{1}}{\beta_{2}}\sqrt{\frac{|z_{2}-a_{2}|}{|z_{1}-a_{2}|}}\frac{|z_{2}-x_{2}|}{|z_{1}-x_{2}|}
=\frac{\beta_{1}}{\beta_{2}}\frac{|z_{2}-x_{2}|}{|z_{1}-x_{2}|}=\Gamma_{2}.$$
From that, we may deduce that $S_{T}$ is a semi-infinite interval
(namely $(-\infty,x_{0})$)
only when $\gamma=\Gamma_{2}$, and that $S_{T}$ is a bounded segment for all values of $\gamma$ in $(\Gamma_{2},1)$, that is during the third phase.
\end{proof}
\subsection{Two limit cases : $\Re z_{1}=\Re z_{2}$ and $z_{2}=1$}
\label{limit-case}
We first consider the limit case
$\Re z_1 = \Re z_2$. Without loss of generality, we may assume $\Re z_2 = \Re z_1 = 0$,
and, thus, $z_1 = \beta_1 i,z_2 = \beta_2 i,$ with $\beta_1, \beta_2 >0.$
Following the analysis carried out in previous Proposition \ref{thm:phasessigned} and Theorem \ref{thm:main}, we are in a position to state the following result.
\begin{theorem}\label{thm:phasessym}
Suppose that the external field $Q$ is given by
\begin{equation*}
Q(x) = \log |x-z_1| - \gamma \log |x-z_2|,\quad z_1 = \beta_1 i,\;z_2 = \beta_2 i,\;\gamma \in (0,1).
\end{equation*}
Then,

\begin{itemize}

\item If $\beta_1 > \beta_2$, the following phases occur

\begin{itemize}

\item {\textbf{Phase 1:}} $\gamma \in (0, \beta_2 / \beta_1]$. Then, $\mu_T = \eta_T$ and $S_T = \R$.


\item {\textbf{Phase 2:}} $\gamma \in (\beta_2 / \beta_1,1)$. Now,
$S_T = (-\infty, -a] \cup [a, +\infty),$ with $a > 0$.

\end{itemize}

\item If $\beta_1 < \beta_2$, we have

\begin{itemize}

\item {\textbf{Phase 1:}} $\gamma \in (0, \beta_1 / \beta_2]$. Then, $\mu_T = \eta_T$ and $S_T = \R$.

\item {\textbf{Phase 2:}} $\gamma \in (\beta_1 / \beta_2,1)$. Now, $S_T = [-a, a],$ with $a > 0$.

\end{itemize}

\end{itemize}

\end{theorem}

\begin{proof}

We assume $\beta_{1}>\beta_{2}$, and obtain our result as a limit case of Theorem \ref{thm:main}. Namely, starting from a configuration as in Fig. \ref
{intro-circle}, we move $z_{1}$ (resp. $z_2$) continuously to the right (resp. left) until they both lie on the imaginary axis. Then the semi-circle is deformed into a vertical half-line connecting $x_{2}$ with $\infty$, while both points $x_{1}$ and $x_{0}$ move to infinity. Moreover the values of $\Gamma_{1}$ and $\Gamma_{2}$, given in \eqref{Gamma1} and \eqref{Gamma2}, respectively, become
$$\Gamma_{1}=\frac{\beta_{2}}{\beta_{1}},\quad\Gamma_{2}=1,$$
meaning that the first two phases are as in Theorem \ref{thm:main}, while the third phase disappears. The endpoints $a_{1}$ and $a_{2}$ during the second phase are opposite, because of symmetry, and tend to infinity as $\gamma$ tends to 1.

When $\beta_{1}<\beta_{2}$, we move again $z_{1}$ and $z_{2}$ continuously to the imaginary axis. Now the semi-circle is deformed into a vertical half-line connecting $x_{1}$ with $\infty$, while both points $x_{2}$ and $x_{0}$ move to infinity. Moreover the values of $\Gamma_{1}$ and $\Gamma_{2}$ become
$$\Gamma_{1}=\Gamma_{2}=\frac{\beta_{1}}{\beta_{2}},$$
meaning this time that the second phase disappears. During the first phase, $S_{T}=\R$, the conjugate roots $b$ and $\bar b$ start from $z_{2}$ and $\bar z_{2}$ respectively, move along the vertical line through $z_{2}$, and reach $\infty$ as $\gamma$ tends to $\Gamma_{1}=\Gamma_{2}$. Then, a gap appears at infinity, and the endpoints of $S_{T}$ are again opposite because of symmetry. As $\gamma$ tends to 1, $S_{T}$ reduces to the point $x_{1}$ ($=0$).
\end{proof}

Concerning the limit case $z_{2}=1$ where the repellent lies on $\R$, we just observe that, starting from the initial configuration as in Fig. \ref{intro-circle}, and letting $z_{2}$ tend to 1, we get that $x_{2}$ tends to 1, while $\Gamma_{1}$ and $\Gamma_{2}$ become
$$\Gamma_{1}=0,\qquad\Gamma_{2}=\frac{\beta_{1}}{|z_{1}-x_{2}|}
=\frac{\beta_{1}}{\sqrt{4+\beta_{1}^{2}}},$$
meaning that the dynamics of the support $S_{T}$ is as in Theorem \ref{thm:main}
except that the first phase is absent (note that, as expected, $S_{T}$ never contains the point $1$ where the repellent is located).
\subsection{Computation of the support $S_T$}
The aim of this section is to compute explicitly the support $S_{T}$ as $\gamma$ varies.
Dividing the second equation by the first one in (\ref{residues}), taking modulus and then squaring, one gets, for any $\gamma$,
\begin{equation}\label{explicit}
\left(\frac{\gamma\beta_{2}}{\beta_{1}}\right)^{2}=\frac{|z_{2}-a_{1}||z_{2}-a_{2}||z_{2}-x_{2}|^{2}}
{|z_{1}-a_{1}||z_{1}-a_{2}||z_{1}-x_{2}|^{2}}.
\end{equation}
Moreover, as was mentioned in Remark \ref{apoll}, the circle $\CC$ is an Apollonius circle for $a_{1}$ and $a_{2}$. Thus, there exists some $m>0$, depending on $\gamma$, such that
\begin{equation}\label{apolon}
|z_{2}-a_{2}|=m|z_{2}-a_{1}|,\qquad|z_{1}-a_{2}|=m|z_{1}-a_{1}|.
\end{equation}
Hence, (\ref{explicit}) can be simplified to anyone of the two equations
\begin{equation}\label{firstmethod}
\frac{\gamma\beta_{2}}{\beta_{1}}=\frac{|z_{2}-a_{i}||z_{2}-x_{2}|}
{|z_{1}-a_{i}||z_{1}-x_{2}|},\qquad i=1\text{ or }2.
\end{equation}
Together with the expression of $x_{2}$ in (\ref{x0x1x2}), these two equations allow one to recover
$a_{1}$ and $a_{2}$ from $\gamma$. Figure \ref{support} shows an example. The evolution of the support is displayed for the case where $\beta_{1}=3$, $\beta_{2}=4$, for $\gamma$ growing from $0$ to $1$.
\begin{figure}[ht]
\begin{center}
\includegraphics[width=9cm]{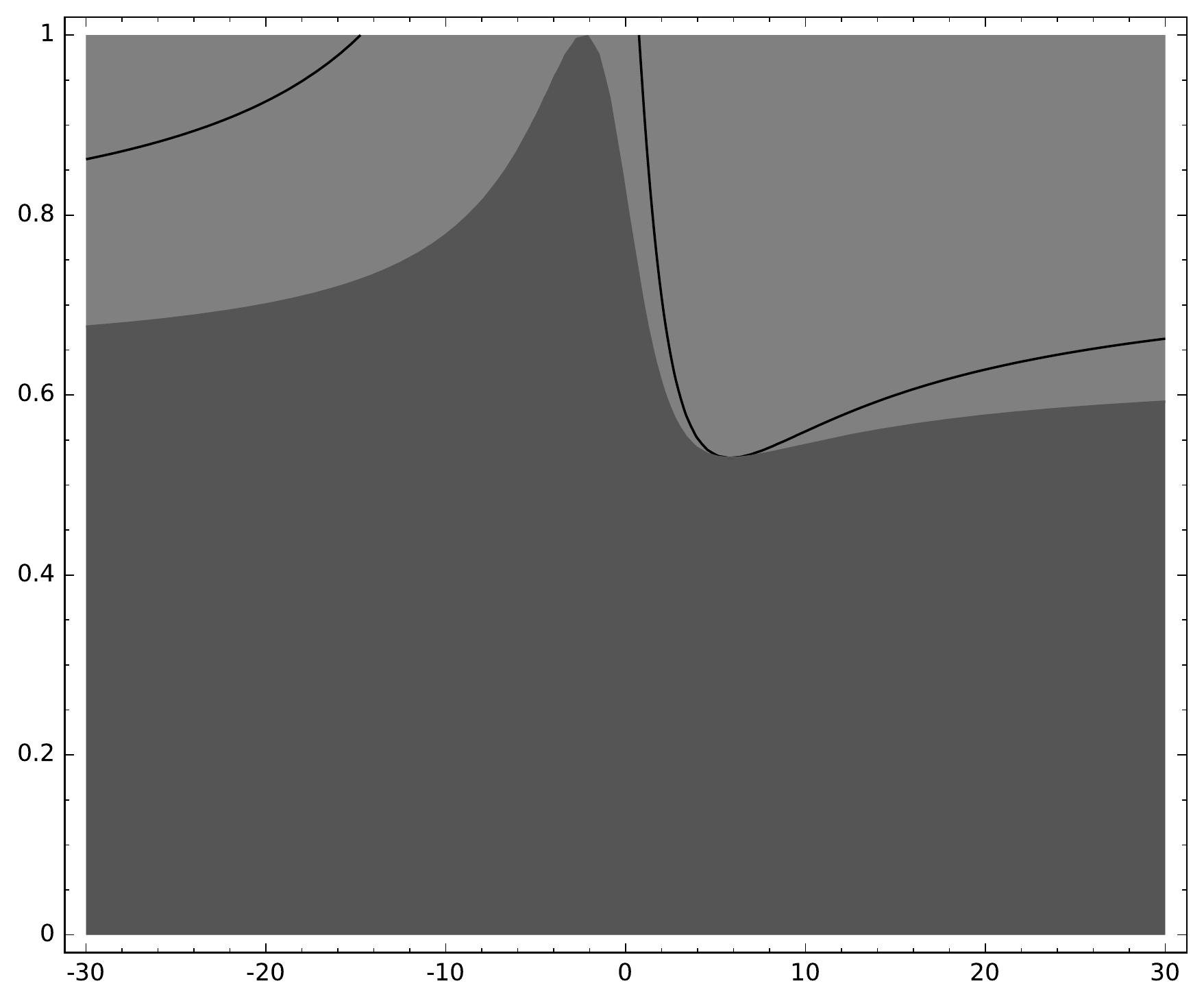}
\caption{Evolution of the support $S_{T}$ as $\gamma$ moves from 0 to 1 ($y$-axis). Here, $\beta_{1}=3$, $\beta_{2}=4$. For small values of $\gamma$, the support is the whole real line, then it becomes the union of two semi-infinite intervals, and finally a bounded interval. When $\gamma=1$, the support reduces to a single point. The black curve shows the endpoints of $S_{T}^{+}$.
For large values of $\gamma$, e.g. $\gamma>0.6$, $S_{T}$ is substantially smaller than $S_{T}^{+}$.}
\label{support}
\end{center}
\end{figure}

In Figure \ref{phase}, we display, for a fixed value of $\gamma=1/2$, the different regions in the positive quadrant of the $(\beta_{1},\beta_{2})$-plane corresponding to different phases. The first phase takes place in the black region, the second phase takes place in the grey region, and the third phase takes place in the white region. In particular, when $\beta_{2}=0$, one sees that the first phase does not occur, as explained at the end of Section \ref{limit-case}.
\begin{figure}[ht]
\begin{center}
\includegraphics[width=9cm]{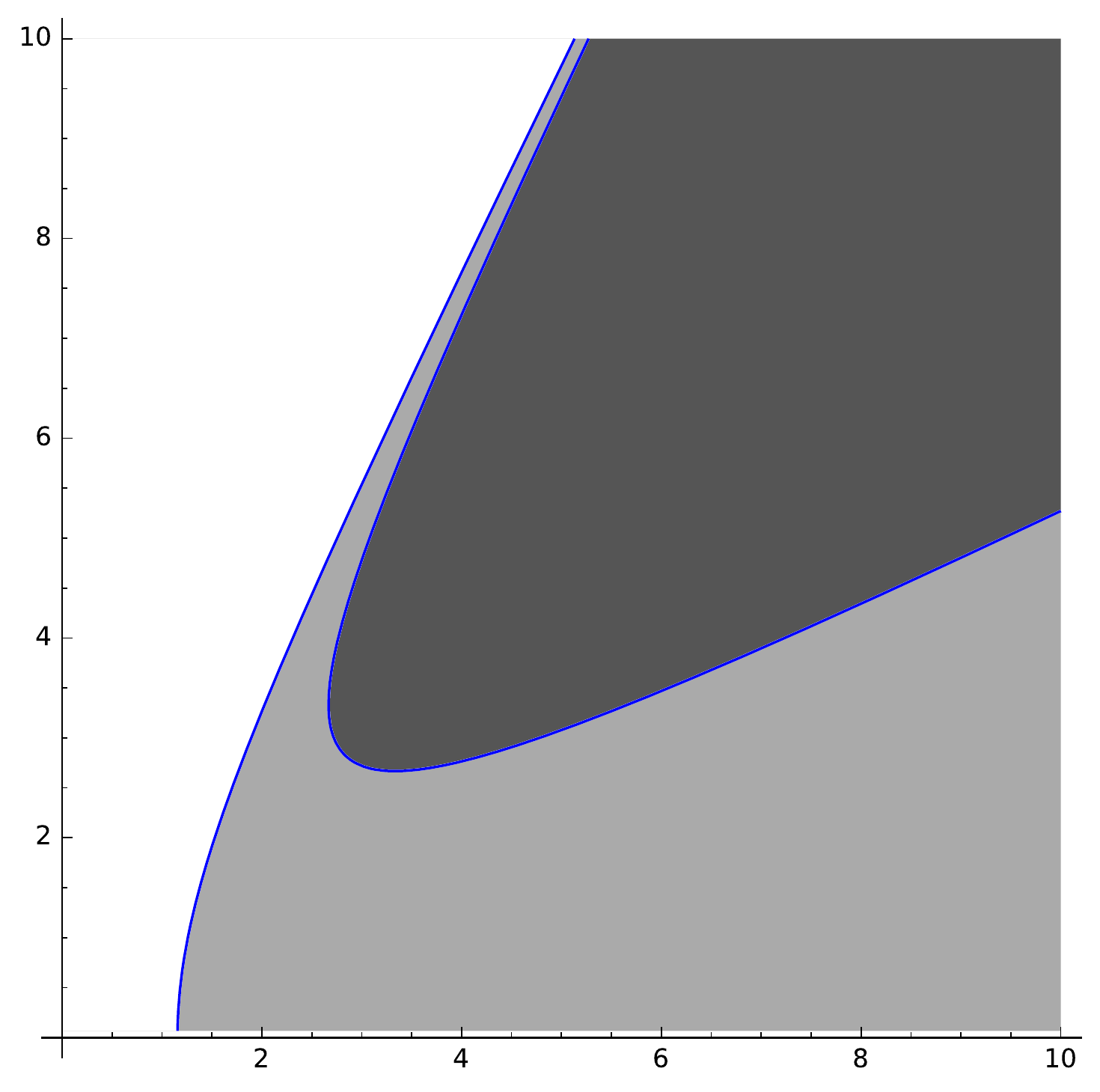}
\caption{Different regions in the positive quadrant of the $(\beta_{1},\beta_{2})$-plane corresponding to different phases when $\gamma=1/2$.}
\label{phase}
\end{center}
\end{figure}

Finally, in the symmetric case described in Theorem \ref{thm:phasessym}, it is easy to see that the density of the equilibrium measure during the second phase is given by
$$
\mu_{T}'(x)=\begin{cases}
{\displaystyle \frac{d|x|\sqrt{x^2-a^2}}{(x^2+\beta_1^2)(x^2+\beta_2^2)}}, \quad
x\in \R \setminus (-a,a)\quad\text{if }\beta_{1}>\beta_{2},\\[10pt]
{\displaystyle\frac{d\sqrt{a^2-x^2}}{(x^2+\beta_1^2)(x^2+\beta_2^2)}},
\quad x\in [-a,a]\quad\text{if }\beta_{1}<\beta_{2},
\end{cases}
$$
with $d$ the normalization constant, and where the endpoint $a$ may be computed, e.g. from (\ref{firstmethod}), as
$$
a = \sqrt{\frac{\gamma^2\,\beta_1^2-\beta_2^2}{1-\gamma^2}}
\quad\text{ if }\beta_{1}>\beta_{2},\text{ and }\quad
a = \beta_1 \beta_2\,\sqrt{\frac{1-\gamma^2}{\gamma^2\,\beta_2^2-\beta_1^2}}
\quad\text{ if }\beta_{1}<\beta_{2}
$$
For instance, if $\beta_1 = 1$ and $\beta_2 = 3$, $\gamma = .5 > {\beta_1}/{\beta_2}$, $S_T = [-a,a]$ with $a = 3 \sqrt{{3}/{5}} \approx 2.32$.
Figure \ref{fig:comparing} shows the densities of $\eta^+_T$ and $\mu_{T}$. Note that $S_{T}^{+}=[-3.87...,3.87...]$ is much bigger than $S_T$.
\begin{figure}[ht]
    \begin{center}
        \includegraphics[scale=0.4]{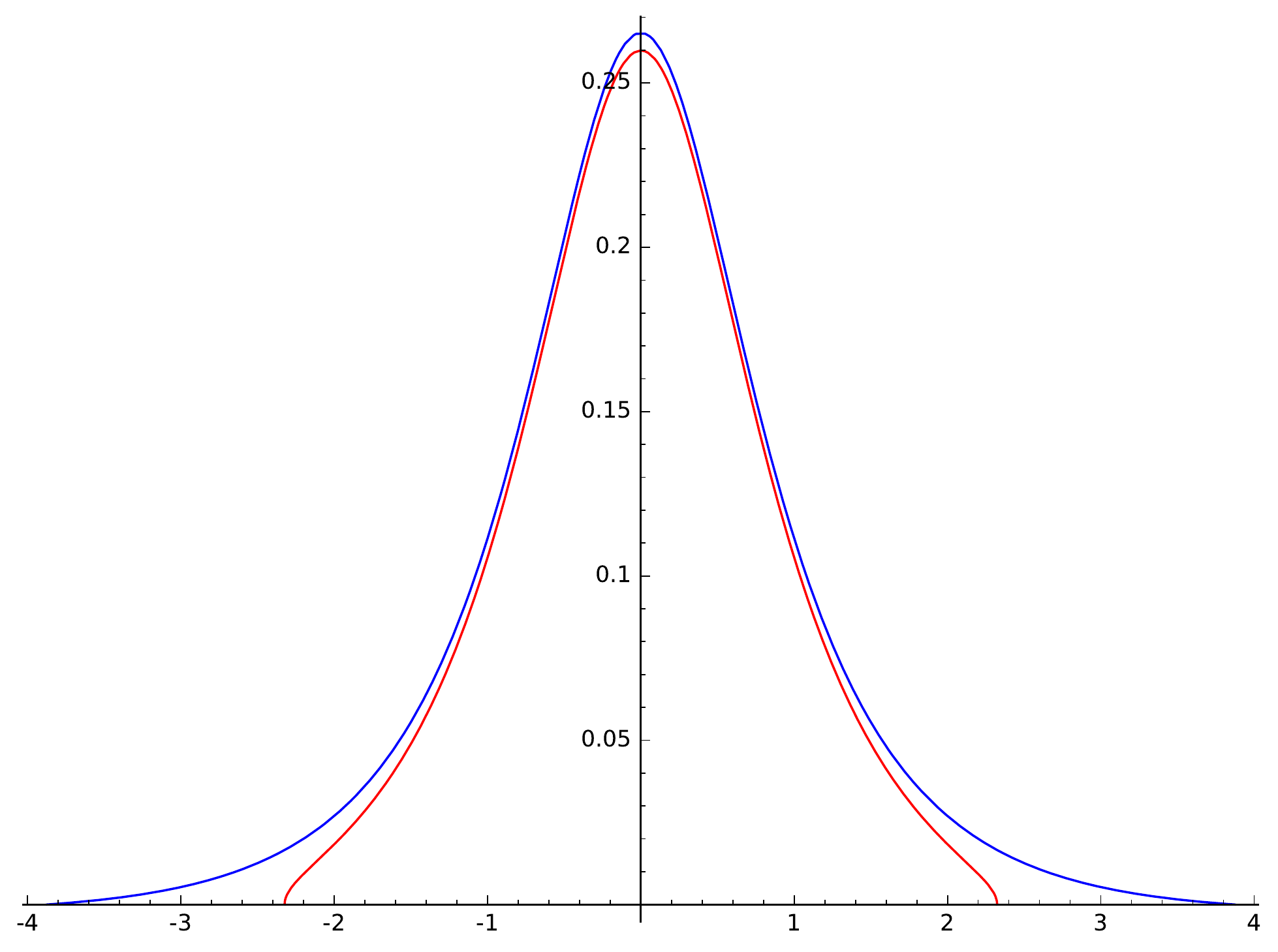}
    \end{center}
    \caption{Densities of $\eta_T^{+}$ and $\mu_T$ in the symmetric case $z_{1}=i, z_2 = 3i$ and $\gamma = .5$.}
    \label{fig:comparing}
\end{figure}

\section{Some results in potential theory} \label{Pot}

We first recall the notion of balayage of measures, and compute the balayage of a pointwise charge on the real axis. We then extend a few classical results from potential theory, usually stated for measures with compact supports, to the case of unbounded supports. In particular, we consider equilibrium measures and signed equilibrium measures in weakly admissible external fields.
We denote by $\Sigma$ a closed subset of $\C$.
\subsection{Balayage of measures}
The notion of balayage of a measure may be found, e.g., in \cite[Chapter IV]{Lan}, \cite[Chapter II.4]{Saff:97} or \cite[Appendix VII]{StT}.
\begin{definition}\label{def:balayage}
Given a closed set $\Sigma\subset \C$ and a measure $\sigma$ with compact support in $\C\setminus \Sigma$, the measure $\widehat{\sigma}:= \Bal (\sigma,\Sigma)$ is said to be the balayage of $\sigma$ onto $\Sigma$ if it has the same mass, $\supp \widehat{\sigma} \subseteq \partial \Sigma$, and
\begin{equation*}
V^{\widehat{\sigma}}(z) = V^{\sigma}(z) + C\quad\text{q.e. on }\Sigma,\qquad V^{\widehat{\sigma}}(z) \leq V^{\sigma}(z) + C\quad\text{on }\C,
\end{equation*}
where $C=0$ if $\C \setminus\Sigma$ is a bounded set.
\end{definition}
We next compute the balayage $\Bal (\delta_z, \R)$ of a positive unit mass at $z\in \C \setminus \R$ onto the real axis. 
\begin{lemma}\label{lem:realbalay}
We have
\begin{equation}\label{zrealbalayage}
d\Bal(\delta_z, \R)(x) = \frac{1}{\pi}\frac{|\Im z|}{|x-z|^2} dx,\quad x\in \R.
\end{equation}
\end{lemma}

\begin{proof}
Assume, without loss of generality, that $\Im(z)>0$, and consider the conformal map $\varphi$ from the Riemann sphere to itself given by
\begin{equation}\label{conform}
\varphi(u):u\mapsto w=2\frac{\Im(z)}{u-z}.
\end{equation}
It maps the upper half-plane $\H$ to the exterior of the disk $D$ of center $i$ and radius $1$, the axis $\R$ to the circle $K=\partial D$, and sends the point $z$ to infinity.
Denote by $\omega(z,\cdot,\H)$ the harmonic measure for the domain $\H$ with pole at $z$, and by $\omega_{K}$ the equilibrium measure for $K$. Then, for $A$, a Borel subset of $\R$, we have
\begin{equation}\label{rel-bal-harm}
\Bal(\delta_z, \R)(A) = \omega(z,A,\H)=\omega(\infty,\varphi(A),\C\setminus\bar D)=\omega_{K}(\varphi(A)),
\end{equation}
where we refer to \cite[Appendix A3]{Saff:97} for the first equality, to the subordination principle for harmonic measures, cf. \cite[Theorem 4.3.8]{Ran}, for the second equality, and to \cite[Theorem 4.3.14]{Ran} for the third equality.
Formula (\ref{zrealbalayage}) follows from (\ref{conform}), (\ref{rel-bal-harm}), and the fact that $\omega_{K}=d\theta/2\pi$.
\end{proof}

\subsection{Measures with unbounded supports}
\label{unbdd}
We first recall four lemmas.
\begin{lemma}[{\cite[Lemma 2.4]{BKMW}}]\label{weak-lsc}
Assume that a bounded sequence of positive measure $\mu_{n}$ tends weakly to $\mu$, and let $Q$ be a lower bounded, lower semi-continuous function on $\C$. Then,
$$\int Qd\mu\leq\liminf_{n\to\infty}\int Qd\mu_{n}.$$
\end{lemma}
\begin{lemma}[{\cite[Lemma 2.2]{BKMW}}]\label{cap-mu}
Let $\mu$ be a finite measure supported on $\Sigma$, of finite energy. Then
$\mu(E)=0$ for every Borel polar set $E$.
\end{lemma}
Note that $\mu$ is not supposed to have compact support, compare with \cite[Theorem 3.2.3]{Ran}.
\begin{lemma}
Let $\mu$ and $\nu$ be two probability measures supported on $\Sigma$, with finite energies. Then
$$I(\nu-\mu)\geq 0.$$
\end{lemma}
The case of measures with compact supports is in \cite[Lemma I.1.8]{Saff:97}. The unbounded case is proven in \cite{CKL}.
\begin{lemma}[{\cite[Lemma 3.2]{BLW}}]
Let $\mu$ be a probability measure supported on an unbounded set $\Sigma$. Then,
\begin{equation}\label{cond-super}
\int\log(1+|t|)d\mu(t)<\infty\quad\Longleftrightarrow\quad
-\infty<I(\mu)\quad\Longleftrightarrow\quad
\exists z_{0}\in\C,~V^\mu(z_{0})>-\infty.
\end{equation}
If these conditions are satisfied, the potential $V^\mu$ is a superharmonic function and
\begin{equation*}
-V^\mu(z)\leq\log(1+|z|)+\int\log(1+|t|)d\mu(t).
\end{equation*}
\end{lemma}
When the measure $\mu$ satisfies the first inequality in (\ref{cond-super}) we say that $\mu$ integrates the logarithm at infinity.

We next give a version of the principle of descent for measures with unbounded supports, compare with \cite[Theorem I.6.8]{Saff:97}.
\begin{theorem}[Principle of descent]\label{th-descent}
Let $\mu_{n}$ be a log-tight sequence of finite positive measures, that is
\begin{equation}\label{log-t}
\forall\epsilon>0,~\exists\text{ compact set }K\subset\C,~\forall n\in\N,\quad
\int_{\C\setminus K}\log(1+|t|)d\mu_{n}(t)\leq\epsilon,
\end{equation}
and assume that $\mu_{n}$ tends weak-* to a finite measure $\mu$. Then,
\begin{equation}\notag
V^\mu(z)\leq\liminf_{n\rightarrow\infty}V^{\mu_{n}}(z),\qquad z\in\C.
\end{equation}
More generally, if $z_{n}$ is a sequence of points tending to a limit $z^{*}\in\C$, then
\begin{equation}\label{descent}
V^\mu(z^{*})\leq\liminf_{n\rightarrow\infty}V^{\mu_{n}}(z_{n}).
\end{equation}
\end{theorem}
\begin{proof}
We prove \eqref{descent}. Let $\epsilon>0$. We choose a closed disk $K$ centered at $z^{*}$ sufficiently large so that, for $z$ in some neighborhood of $z^{*}$,
$$
\forall n\in\N,\quad
\int_{\C\setminus K}\log(|z-t|)d\mu_{n}(t)\leq\epsilon,
$$
which is possible, in view of \eqref{log-t} and the inequality $|z-t|\leq(1+|z|)(1+|t|)$.
Moreover, by the finiteness assumption of the measure $\mu$, we may choose $K$ such that $\mu (\partial K) = 0$, where $\partial K$ denotes the boundary of $K$. Then, by \cite[Theorem 0.5']{Lan}, the restrictions $\mu_{n|K}$ tend weak-* to $\mu_{|K}$. By the principle of descent for compactly supported measures, we have
$$V^{\mu_{|K}}(z^{*})\leq\liminf_{n\rightarrow\infty}V^{\mu_{n|K}}(z_{n}).
$$
Furthermore, assuming that $K$ is large enough so that $\log|z^{*}-t|\geq0$ for $t\notin K$, we have
\begin{align*}
V^{\mu}(z^{*}) & \leq V^{\mu_{|K}}(z^{*})\leq\liminf_{n\rightarrow\infty}
\left(V^{\mu_{n}}(z_{n})+\int_{\C\setminus K}\log(|z_{n}-t|)d\mu_{n}(t)\right)\\
& \leq
\liminf_{n\rightarrow\infty}
V^{\mu_{n}}(z_{n})+\epsilon.
\end{align*}
Letting $\epsilon$ go to zero, we obtain \eqref{descent}.
\end{proof}
Assume now that $Q$ is a weakly admissible weight on an unbounded set
$\Sigma$.
It is known from \cite[Theorem 3.4]{BLW} that the potential of the weighted equilibrium measure $\mu_{Q}$ still satisfies the Frostman inequalities \eqref{Frost1}--\eqref{Frost2}.
In the next proposition, we verify that such inequalities also characterize $\mu_{Q}$, as is the case for admissible weights.
\begin{proposition}\label{char-Frost}
Assume $\mu$ is a measure supported on $\Sigma$, with finite energy $I(\mu)$, and there exists a (finite) constant $C$ such that
\begin{align*}
V^\mu(z)+Q(z)\geq C & \quad\text{q.e. on }\Sigma,\\
V^\mu(z)+Q(z)\leq C & \quad\text{q.e. on }S_{\mu},
\end{align*}
where $Q$ is a weakly admissible weight on $\Sigma$. Then $\mu=\mu_{Q}$.
\end{proposition}
\begin{proof}
Let $\nu$ be a measure on $\Sigma$ with finite energy $I(\nu)$. Integrating the first inequality with respect to $\nu$ and the second with respect to $\mu$ shows that
$$\int (V^\mu+Q)d\nu\geq\int (V^\mu+Q)d\mu.$$
Now, we have
\begin{equation}\label{ineq}
I_{Q}(\nu)=I_{Q}(\mu+(\nu-\mu))=I_{Q}(\mu)+2\int(V^\mu+Q)d(\nu-\mu)+I(\nu-\mu)\geq I_{Q}(\mu),
\end{equation}
where the sum in the third expression is well-defined (meaning there is no undetermined form $\infty-\infty$ in it) since $I_{Q}(\mu)>-\infty$, $\int(V^\mu+Q)d\mu$ equals the finite constant $C$, $\int(V^\mu+Q)d\nu\geq C$ and $I(\nu-\mu)\geq0$. Since the inequality (\ref{ineq}) is satisfied for any measure $\nu$ of finite energy, one derives that $\mu=\mu_{Q}$ (recall from \cite[Theorem 3.4]{BLW} that a measure with infinite energy cannot be the equilibrium measure).
\end{proof}
\subsection{A version of de La Vall\'ee Poussin theorem for measures with unbounded supports}
We give an extended version of ``de La Vall\'ee Poussin'' theorem. The classical version assumes that the measures have compact supports, see e.g. \cite[Theorem IV.4.5]{Saff:97}.
\begin{theorem}\label{dLVP}
Let $\mu$ and $\nu$ be two measures supported in $\C$, which integrate the logarithm at infinity, satisfying
$$(\supp\mu)\cup(\supp\nu)\neq\C,$$
and let $\Omega$ be a domain where $V^\mu$ and $V^{\nu}$ are finite and satisfy, for some constant $c$,
\begin{equation}\label{LVP}
V^\mu(z)\leq V^{\nu}(z)+c,\qquad z\in\Omega.
\end{equation}
Then, if $A\subset\Omega$ is the subset where equality holds, we have $\nu_{|A}\leq\mu_{|A}$.
\end{theorem}
\begin{proof}
We may assume, without loss of generality, that
\begin{equation}\label{assump}
0\notin(\supp\mu)\cup(\supp\nu).
\end{equation}
The idea of proof is to map the problem to a compact setting by using the inversion
\begin{equation}\label{invers}
L:\C\setminus\{0\}\to\C\setminus\{0\},\qquad L(z)=1/z.
\end{equation}
When a subset $K$ of $\C$ is unbounded, we make the convention that
$$L(K)=\{L(x),~x\in K\}\cup\{0\},$$
For $\nu$ a Borel measure supported on the set $K$, we denote by $L_{*}\nu$ its push-forward by $L$, that is, the measure on $L(K)$ such that
$$\int_{L(K)}f(x)dL_{*}\nu(x)=\int_{K}f(L(x))d\nu(x),$$
for any Borel function $f$ on $L(K)$. For $K$ unbounded, one may check that the map
$$L_{*}:\MM(K)\to\MM(L(K))$$
is a homeomorphism from $\MM(K)$ to the subset of $\MM(L(K))$ of measures which put no mass at $0$.
For a measure $\eta$ supported on a set $K$,
\begin{align}
V^{\eta}(z) & = \int\log\frac{1}{|z-t|}d\eta(t)=\int\log\frac{1}{|tz||L(z)-L(t)|}d\eta(t)
\notag\\
& =V^{L_{*}\eta}(L(z))-\eta(\C)\log|z|-\int\log|t|d\eta(t)\notag\\
\label{corresp-U}
& =V^{L_{*}\eta}(L(z))-\eta(\C)V^{\delta_0}(L(z))-c_{\eta},
\end{align}
where the constant $c_{\eta}:=\int\log|t|d\eta(t)$ is finite if $0\notin\supp\eta$ and $\eta$ integrates the logarithm at infinity.
Note that, in view of assumption (\ref{assump}), we may restrict ourselves to a domain $\Omega$ that does not contain 0.
The inequality (\ref{LVP}) translates into
\begin{equation}\label{ineg-T}
V^{L_{*}\mu+\nu(\C)\delta_0}(z)\leq V^{L_{*}\nu+\mu(\C)\delta_0}(z)+c-c_{\nu}+c_{\mu},\qquad z\in L(\Omega)\setminus\{0\},
\end{equation}
Note that both potentials are finite on $L(\Omega)\setminus\{0\}$. Applying the version of Theorem \ref{dLVP} for compactly supported measures to $L_{*}\mu+\nu(\C)\delta_0$ and $L_{*}\nu+\mu(\C)\delta_0$ with the domain $L(\Omega)\setminus\{0\}$, we obtain
$L_{*}\mu_{|A}\leq L_{*}\nu_{|A}$, with
$A$ the subset of $L(\Omega)\setminus\{0\}$ where equality holds in (\ref{ineg-T}).  It then suffices to apply the map $L$ to the last inequality to finish the proof.
\end{proof}
\subsection{Domination principle and the signed equilibrium measure}
\label{sgn-meas}
We first give a version of the domination principle for logarithmic potentials of measures with possibly unbounded supports, compare with the corresponding statement in \cite[Section I.3]{Saff:97}.
\begin{theorem}[Domination Principle]\label{dom}
Let $\mu$ and $\nu$ be positive and finite measures with closed supports in
$\C$, which integrate the logarithm at infinity, and such that
$$\nu(\C)\leq\mu(\C)\qquad\text{and}\qquad I(\mu)<\infty.$$
Then,
$$V^\mu(z)\leq V^{\nu}(z)+C,\quad\mu-a.e.\quad\implies\quad
V^\mu(z)\leq V^{\nu}(z)+C,\quad z\in\C,$$
where $C$ is some constant.
\end{theorem}
\begin{proof}
1) We first assume that
$\supp\mu\neq\C$.
\\
-- Case 1.a) : $\supp\mu\cup\supp\nu\neq\C$. Without loss of generality, one may suppose that
 $B(0,R)\subset\C\setminus(\supp\mu\cup\supp\nu)$.
Making use of the inversion (\ref{invers}) and the relation (\ref{corresp-U}) between potentials, the inequality which holds $\mu$-a.e. becomes
$$
V^{L_{*}\mu}(z)\leq V^{L_{*}\nu+(\mu(\C)-\nu(\C))\delta_0}(z)+C-c_{\nu}+c_{\mu},\qquad L_{*}\mu-a.e.,
$$
where the constants
$$c_{\nu}=\int\log|t|d\nu(t),\qquad c_{\mu}=\int\log|t|d\mu(t),$$
are finite.
Applying the domination principle with the positive measures $L_{*}\mu$ and $L_{*}\nu+(\mu(\C)-\nu(\C))\delta_0$ which have compact supports (note also that
$I(L_{*}\mu)=I(\mu)+2\mu(\C)c_{\mu}$ is finite and the inequality on the masses is satisfied), and then making use of $L$ again, one obtains
$V^\mu(z)\leq V^{\nu}(z)+C$ everywhere on $\C$.\\
-- Case 1.b) : $\supp\mu\cup\supp\nu=\C$. We may still assume that $B(0,R)\subset\C\setminus(\supp\mu)$. Then $B(0,R)\subset\supp\nu$.
The finite measure $\nu$ has at most a countable number of mass points, hence there is some
$z\in B(0,R)$ such that $\nu(\{z\})=0$. By translation, (and possibly choosing a smaller
value of $R$) one may assume
 $z=0$. Let $\nu_{n}=\nu|_{B_{n}}$, $B_{n}:=B(0,1/n)$, and $\tilde\nu_{n}=\nu-\nu_{n}$. One has, for $|z|>R$,
$$
V^{\nu_{n}}(z)\leq\epsilon_{n}:=\log\left(\frac{1}{R-1/n}\right)\nu_{n}(\C)\to 0,\quad\text{as }n\to\infty,
$$
where $\nu_{n}(\C)\to0$ as $n\to\infty$ by monotone convergence. The inequality that holds $\mu$-a.e. implies that
$$V^\mu(z)\leq V^{\tilde\nu_{n}}(z)+\epsilon_{n}+C,\qquad\mu-a.e.,$$
with $\tilde\nu_{n}(\C)\leq\mu(\C)$.
The result obtained in the first case gives
$$V^\mu(z)\leq V^{\nu}(z)-V^{\nu_{n}}(z)+\epsilon_{n}+C,\qquad\text{everywhere in }\C.$$
Now, it suffices to let $n$ go to infinity. Indeed, for a given $z\neq0$, $V^{\nu_{n}}(z)\to0$ by the dominated convergence theorem.
For $z=0$, if $V^{\nu}(0)<\infty$, we still have $V^{\nu_{n}}(0)\to0$ by dominated convergence, and if $V^{\nu}(0)=\infty$, the sought inequality is obviously satisfied.
2) We now assume that $\supp\mu=\C$.
Let $\epsilon>0$. 
Let $\mu_{n}=\mu|_{\C\setminus B_{n}}$, $B_{n}:=B(0,n)$. One has, for $z\in B(0,n)$, 
\begin{align*}
\int\log|z-t|d\mu_{n}(t) & \leq\int\log(1+|z|)d\mu_{n}(t)+\int\log(1+|t|)d\mu_{n}(t)\\
& \leq2\int\log(1+|t|)d\mu_{n}(t)\leq\epsilon,
\end{align*}
for $n$ large enough. Next, denote by $A$ the subset of $\C$ of full $\mu$-measure, i.e. $\mu(\C\setminus A)=0$, where the inequality $V^{\mu}(z)\leq V^{\nu}(z)+C$ holds. Since $I(\mu)<\infty$, $A$ is of positive capacity. In particular, there exists some compact set $K\subset A$ of positive capacity. Denote by $\omega_{K}$ its equilibrium measure, and recall that $V^{\omega_{K}}(z)\leq I_{K}$, $z\in\C$, where $I_{K}$ is the energy of $K$. Now, we define
$$
\tilde\mu_{n}=\mu-\mu_{n}+\mu_{n}(\C)\omega_{K}.
$$
Then, $\tilde\mu_{n}$ has the same mass as $\mu$, is of finite energy, and since the support of $\omega_{K}$ is included in $A$, we have
$$
V^{\tilde\mu_{n}}(z)\leq V^{\nu}(z)+\epsilon+\mu_{n}(\C)I_{K}+C,\quad\tilde\mu_{n}-a.e.
$$
By applying the result of Case 1, we derive that the above inequality is satisfied everywhere in $\C$. Finally, $\tilde\mu_{n}$ is a log-tight sequence of measures that tends to $\mu$. Hence, by letting $\epsilon$ tend to 0, $n$ tend to infinity, and applying Theorem \ref
{th-descent}, we get
$$
V^{\mu}(z)\leq\liminf_{n\to\infty}V^{\tilde\mu_{n}}(z)\leq V^{\nu}(z)+C,\quad z\in\C.
$$
\end{proof}
Let $Q$ be a weakly admissible external field on a closed subset $\Sigma$ of the complex plane.
Recall that the notion of a {\it signed equilibrium measure} $\eta_{T}$ was defined in Section 2. We mentioned there that, in the presence of an admissible weight, the support of its positive part $\eta_t^+$ (corresponding to the Jordan decomposition) contains the support of the equilibrium distribution. We now prove that this result remains valid in the weakly admissible case as well.

\begin{proposition}
Suppose $\sigma=\sigma^{+}-\sigma^{-}$ is a signed measure on the possibly unbounded set $\Sigma\neq\C$, such that
\\
i) $\sigma^{+}(\C)-\sigma^{-}(\C)=1$,\\
ii) the measures $\sigma^{+}$ and $\sigma^{-}$ have finite energies,\\
iii) $V^{\sigma^{+}}$ and $V^{\sigma^{-}}$ are finite on $\supp\sigma$,\\
iv) there exists a constant $C$ such that
\begin{equation}\label{sign-meas}
V^{\sigma}(z)+Q(z)= C  \quad\text{q.e. on }\Sigma,
\end{equation}
Then,
$\mu_{Q}\leq\sigma^{+}$. In particular, the support of $\mu_{Q}$ is contained in the support of $\sigma^{+}$.
\end{proposition}
\begin{proof}
From Frostman inequalities and (\ref{sign-meas}), one has
\begin{align}
V^{\mu_{Q}+\sigma^{-}}(z)+C-F_{Q}\geq V^{\sigma^{+}}(z), & \quad\text{q.e. on }\Sigma,\label{ineg1}\\
V^{\mu_{Q}+\sigma^{-}}(z)+C-F_{Q} = V^{\sigma^{+}}(z), & \quad\text{q.e. on }
\supp(\mu_{Q}). \notag
\end{align}
We want to apply the domination principle (Theorem \ref{dom}) to the first inequality. First, note that it holds $\sigma^{+}$-a.e. since $\sigma^{+}$ has finite energy (recall Lemma \ref{cap-mu}). Note also that $(\mu_{Q}+\sigma^{-})(\C)=\sigma^{+}(\C)$.
Hence the domination principle applies and (\ref{ineg1}) holds everywhere in $\C$. Then it suffices to make use of Theorem \ref{dLVP} with $\Omega=\C$, where we remark that all three potentials $V^{\mu_{Q}}, V^{\sigma^{-}}, V^{\sigma^{+}}$ are finite in $\C$.
\end{proof}
\section*{Acknowledgements}
R.O. was partially supported by the Research Project of Ministerio de Ciencia e Innovaci\'on (Spain) under grant MTM2015-71352-P.
J.S.L. was partially supported by the Research Projects of Ministerio de Ciencia e Innovaci\'on (Spain), under grant MTM2015-71352-P, and Junta de Andaluc\'ia, under grant FQM384.

The authors thank the referees for their valuable suggestions and criticisms which have improved the manuscript.

R.O. also wishes to thank Prof. Peter Dragnev (PFW) for helpful discussions about the signed equilibrium measures during his stays in PFW, Indiana (USA).

\end{document}